\documentclass[11pt]{article}

\usepackage{amsfonts}
\usepackage{amssymb}
\usepackage{amsmath}
\usepackage{fullpage}
\usepackage{color}
\usepackage{amsthm}
\usepackage{dsfont}
\usepackage{graphicx}

\newtheorem{theorem}{Theorem}
\newtheorem{proposition}[theorem]{Proposition}
\newtheorem{corollary}[theorem]{Corollary}
\newtheorem{lemma}[theorem]{Lemma}
\newtheorem{remark}[theorem]{Remark}
\newtheorem{definition}[theorem]{Definition}

\numberwithin{theorem}{section}
\numberwithin{equation}{section}

\newcommand{\dlambda}{\, {\rm d} \lambda}
\newcommand{\dx}{\, {\rm d} x}

\newcommand{\EE}{\mathbb{E}}
\newcommand{\VV}{\mathbb{V}}
\newcommand{\R}{\mathbb{R}}

\newcommand{\bmu}{\boldsymbol{\mu}}
\newcommand{\bnu}{\boldsymbol{\nu}}

\definecolor{darkgreen}{rgb}{0,0.7,0}

\begin{document}
\title{Quasi-Monte Carlo and Multilevel Monte Carlo Methods for Computing Posterior Expectations in Elliptic Inverse Problems}
\author{R. Scheichl, A.M. Stuart, A.L. Teckentrup}
\date{}

\maketitle

\begin{abstract}
We are interested in computing the expectation of a functional of a PDE solution under a Bayesian posterior distribution. Using Bayes' rule, we reduce the problem to estimating the ratio of two related prior expectations. For a model elliptic problem, we provide a full convergence and complexity analysis of the ratio estimator in the case where Monte Carlo, quasi-Monte Carlo or multilevel Monte Carlo methods are used as estimators for the two prior expectations. We show that the computational complexity of the ratio estimator to achieve a given accuracy is the same as the corresponding complexity of the individual estimators for the numerator and the denominator. We {also include numerical simulations,
in the context of the model elliptic problem, which demonstrate the 
effectiveness of the approach.} 
\end{abstract}

\section{Introduction}
{ Simulation frequently plays an essential role in the mathematical modelling of physical processes. However, the model parameters are often subject to uncertainty. This may be due to incomplete or inaccurate knowledge of the system or due to an inherent variability. It is important to understand how this uncertainty in the input parameters influences the reliability of the simulation outputs.}

In the Bayesian framework, we initially assign a probability distribution, called the {\it prior distribution}, to the input parameters. {In addition {\it observations}, related to the model outputs, are often available, and it is then possible to reduce the overall uncertainty and get a better representation of the input parameters by conditioning the prior distribution on this data. This leads to the {\it posterior} distribution on the input parameters. The goal of the simulations is then often to compute the expected value of a 
{\em quantity of interest} (related to the model outputs) under the posterior 
distribution. This is the problem of {\em Bayesian inference}.}

Typically, the posterior distribution is intractable, in the sense that direct sampling is unavailable.
One way to circumvent this problem is to use a Markov chain Monte Carlo (MCMC) approach to sample from the posterior distribution \cite{robert_casella,kst13,crsw13,hss13,cmps14}. However, for large-scale applications where the number of input parameters is typically large and the solution of the forward model expensive, MCMC methods require careful tuning and may become infeasible in practice. 

{An alternative approach, considered in this paper, is to note that any given posterior
expectation can be written as the ratio of two prior expectations,
both involving the {\em likelihood}. The denominator is the normalising constant in Bayes' rule, namely the expected value under the prior of the likelihood.
The numerator is similar, but the likelihood is weighted by the test function
of interest.} This approach has already been considered in \cite{ss12,ss13,ss14}, where the prior expectations are computed using adaptive sparse grid techniques. Similar ideas are also found in \cite{apss15} in the context of importance sampling. The focus of this work is to use sampling methods to compute the prior expectations, which are also well-suited to the case of high dimensional inputs. In particular, we investigate the use of Monte Carlo (MC), quasi-Monte Carlo (QMC) \cite{niederreiter,ks05,kss12,gknsss15} and multilevel Monte Carlo (MLMC) \cite{giles08,heinrich01,bsz11,tsgu13,anst14,ehm14} methods. { The work is closely related to the independent works \cite{dggs16,dggs16_2,gp16} that also investigate the use of QMC methods in computing posterior expectations. The papers \cite{dggs16,dggs16_2} are narrower in the range of problems considered. Their analysis is based on holomorphy arguments which require uniformly bounded coefficients, thus excluding the Gaussian case considered here. On the other hand, the range of methods considered is wider and includes higher-order QMC and multilevel QMC methods. The work \cite{gp16} investigates the use of higher order QMC methods for computing posterior expectations arising from partial differential equations posed on random domains, again in the case of bounded parameters.}

As a particular example, we consider the model inverse problem of determining the
distribution of the diffusion coefficient of a {divergence form}
elliptic partial differential equation (PDE) from observations of a finite set of noisy continuous functionals of the solution. The coefficient distribution is assumed to be determined by an infinite number of scalar parameters through a basis expansion. In contrast to the works \cite{ss12,ss13,ss14,dggs16}, our analysis includes also results in the technically demanding case of log-normal diffusion coefficients, where the differential operator depends in a non-affine way on the parameters, each of which is modelled as a Gaussian random variable under the prior distribution. We provide a full convergence and complexity analysis of the estimator of the posterior expectation in the case of MC, QMC and MLMC 
sampling. {We also demonstrate the effectiveness of this approach for the
estimation of a typical quantity of interest derived from the elliptic
inverse problem. The main conclusion of our work} is that, for a given accuracy, the cost of computing the posterior expectation with any of these Monte Carlo variants is proportional to the computational complexity of the same estimator for prior expectations.  

The remainder of this paper is organised as follows. Section \ref{sec:invprob} provides the mathematical set-up of the inverse problem of interest, including the formulation of ratio estimators for posterior expectations. Section \ref{sec:fe} is then devoted to the analysis of the error committed by approximating the governing equations by finite elements, and Section \ref{sec:samp} introduces MC, QMC and MLMC estimators together with bounds on their sampling errors, extending the QMC analysis to non-linear functionals. In Section \ref{sec:mse}, we then provide a full convergence and complexity analysis of ratio estimators of posterior expectations. We demonstrate the performance of the proposed ratio estimators on a {specific quantity of interest} in Section \ref{sec:num}, and finally provide some conclusions in Section \ref{sec:conc}.

\section{Bayesian Inverse Problems}\label{sec:invprob}
Let $X$ and $V$ be separable Banach spaces, 
and define the Borel measurable mappings $\mathcal{G}: X \rightarrow V$ and $\mathcal H : V \rightarrow \mathbb R^m$, for some $m \in \mathbb N$. We will refer to $\mathcal{G}$ as the {\em forward map} and to $\mathcal H$ as the {\em observation operator}. We denote by $\mathcal F: X \rightarrow \mathbb R^m$ the composition of $\mathcal H$ and $\mathcal{G}$, and by $|\cdot|$  the Euclidean norm on $\mathbb R^m$. 
The inverse problem of interest is to determine the unknown function $u \in X$ from the noisy {\em observations} (or {\em data}) $y \in \R^m$ given by
\begin{equation}
y = \mathcal H (\mathcal{G}(u)) + \eta,
\end{equation}
where the noise $\eta$ is a realisation of the $\mathbb R^m$-valued Gaussian random variable $\mathcal N(0,\Gamma)$, for some (known) covariance matrix $\Gamma$. For simplicity, we will assume that $\Gamma = \sigma_\eta^2 I$, for some positive constant $\sigma_\eta^2$. 

We adopt a Bayesian perspective in which, in the absence of data, $u$ is distributed according to a prior measure $\mu_0$. Under the conditions given in
{Proposition} \ref{thm:rad_nik} below, the posterior distribution $\mu^y$ on the conditioned random variable $u | y$ is absolutely continuous with respect to $\mu_0$ and given by an infinite dimensional version of Bayes' Theorem. This takes the form
\begin{equation}\label{eq:rad_nik}
\frac{d\mu^y}{d\mu_0}(u) = \frac{1}{Z} \theta(\mathcal{G}(u)),
\end{equation}
where
\begin{equation}\label{eq:def_theta}
\theta(\zeta) = \exp[-\Phi(\zeta)], \quad \Phi(\zeta) = \frac{1}{2 \sigma_\eta^2} \left| y -  \mathcal H (\zeta) \right|^2  \quad \text{and} \quad Z = \EE_{\mu_0}[\theta(\mathcal{G}(u))].
\end{equation}
The following {proposition from \cite{stuart10}}
provides conditions under which the posterior distribution $\mu^y$ is well defined and satisfies \eqref{eq:rad_nik}.

\begin{proposition} \label{thm:rad_nik}
Assume the map $\mathcal F : X \rightarrow \mathbb R^m$ is continuous and $\mu_0(X) = 1.$ Then the posterior distribution $\mu^y$ is absolutely continuous with respect the prior distribution $\mu_0$, with Radon-Nikodym derivative given by \eqref{eq:rad_nik}.
\end{proposition}

In applications, it is often of interest to compute the expectation of a functional $\phi : V \rightarrow \mathbb R$ of $\mathcal{G}(u)$ under the posterior distribution $\mu^y$. If we define
\begin{equation}\label{eq:def_psi}
\psi(\zeta) = \theta(\zeta) \, \phi(\zeta) \qquad \text{and} \qquad Q = \EE_{\mu_0}[\psi(\mathcal{G}(u))],
\end{equation}
it follows from \eqref{eq:rad_nik} that the posterior expectation of $\phi(\mathcal{G}(u))$ can be written as
\begin{equation}\label{eq:rat}
\EE_{\mu^y}[\phi(\mathcal{G}(u))] = \frac{\EE_{\mu_0}[\psi(\mathcal{G}(u))]}{\EE_{\mu_0}[\theta(\mathcal{G}(u))]} = \frac{Q}{Z}.
\end{equation}
We will approximate $\EE_{\mu^y}[\phi(\mathcal{G}(u))]$ by using different Monte Carlo type methods to compute the prior expectations $Z$ and $Q$.

\subsection{Parametrisation of the Unknown Input}\label{ssec:invprob_para}
We consider the setting where the Banach space $X$ is a space of real-valued functions defined on a bounded spatial domain $D \subset \mathbb R^d$, for some dimension $d = 1, 2$ or $3$. For ease of presentation, we shall restrict our attention to the case $X = C(\overline D)$, the space of continuous functions on $\overline D$, but other choices are possible (cf. Remark \ref{rem:piecewise}), as is the extension to vector-valued functions.

We assume that the unknown function $u \in X$ admits a parametric representation of the form
\begin{equation}\label{eq:def_upara}
u(x) = m_0(x) + \sum_{j=1}^\infty u_j \phi_j(x),
\end{equation}
where $m_0 \in X$, $\{\phi_j\}_{j=1}^\infty$ denotes an infinite sequence in 
$X$ {(typically {normalised} to one in $X$ or in a larger space containing $X$)} and $\{u_j\}_{j=1}^\infty \subset \mathbb R^\infty$ denotes a set of real-valued coefficients. By randomising the coefficients $\{u_j\}_{j=1}^\infty$, we create real-valued random functions on $\overline D$. 
To this end, we introduce the deterministic, monotonically non-increasing sequence $\gamma = \{\gamma_j\}_{j=1}^\infty$ and the i.i.d random sequence $\xi = \{\xi_j\}_{j=1}^\infty$, and set $u_j = \gamma_j \, \xi_j$. To emphasise the dependence of $u$ on $\xi$, we will write $u = u(x; \xi)$. 

We will consider two specific examples of the infinite series representation \eqref{eq:def_upara}, referred to as uniform priors and Gaussian priors, respectively.

\subsubsection{Uniform Priors}\label{ssec:invprob_para_unif}
In the case of uniform priors, we specify the { i.i.d.} sequence of random variables $\xi = \{\xi_j\}_{j=1}^\infty$ by choosing {$\xi_j \sim U[-1,1]$}, a uniform random variable on $[-1,1]$, and the deterministic sequence $\gamma$ is chosen absolutely summable, $\gamma \in \ell^1 (\R^\infty)$. The functions $\{\phi_j\}_{j=1}^\infty$ and $m_0$ are chosen as elements of $C(\overline D)$, and are assumed normalised so that $\|\phi_j\|_{C(\overline D)}=1$, for all $j \in \mathbb N$. 

We then have the following result from \cite{ds14}.

\begin{lemma}\label{lem:unif_bound} Suppose there are finite, strictly positive constants $m_\mathrm{min}, m_\mathrm{max}$ and $r$ such that
\begin{equation*}
\min_{x \in \overline D} m_0(x) \; \geq\; m_\mathrm{min}, \quad \max_{x \in \overline D} m_0(x) \;\leq\; m_\mathrm{max} \quad \text{and} \quad \|\gamma\|_{\ell^1} = \frac{r}{1+r} m_\mathrm{min}.
\end{equation*}
Then the following holds almost surely: the function $u(\cdot; \xi)$ defined in \eqref{eq:def_upara} is in $C(\overline D)$ and satisfies the bounds
\begin{equation*}
\frac{1}{1+r} m_\mathrm{min} \leq u(x ; \xi) \leq m_\mathrm{max} + \frac{r}{1+r} m_\mathrm{min}, \quad \text{for almost all } \; x \in D.
\end{equation*}
\end{lemma}

Note in particular that the upper and lower bounds on $u$ in Lemma \ref{lem:unif_bound} are independent of the particular realisation of the random sequence $\xi$. With $X = C(\overline D)$, it follows from Lemma \ref{lem:unif_bound} that $\mu_0(X)=1$.
We furthermore have the following result from \cite{ds14} on the spatial regularity of the function $u$ in the case where the functions $(m_0, \{\phi_j\}_{j=1}^\infty)$ are H\"older continuous. 

\begin{lemma}\label{lem:unif_holder} Suppose $m_0$ and $\{\phi_j\}_{j \ge 1}$, are in $C^\alpha(\overline D)$, the space of H\"older continuous functions with exponent $\alpha \leq 1$, and suppose $\sum_{j=1}^\infty |\gamma_j|^2 \|\phi_j\|_{C^\alpha(\overline D)}^\beta < \infty$, for some $\beta \in (0,2)$. Then the function $u(\cdot ; \xi)$ defined in \eqref{eq:def_upara} is in $C^t(\overline D)$ almost surely, for any $t < \alpha \beta / 2$.
\end{lemma}

\subsubsection{Gaussian Priors}\label{ssec:invprob_para_gauss}

For Gaussian priors, we specify the i.i.d sequence of random variables $\xi$ by choosing {$\xi_j \sim N(0,1)$}, a standard Gaussian random variable with mean 0 and variance 1. 
{ We choose the sequences $\{\phi_j\}_{j=1}^\infty$ and $\{\gamma_j^2\}_{j=1}^\infty$ to be the eigenfunctions and eigenvalues, respectively, of a covariance operator $\mathcal C :L^2(D) \rightarrow L^2(D)$, such that the series \eqref{eq:def_upara} is the Karhunen-Loeve (KL) expansion of the Gaussian measure $\mu_0 = N(m_0, \mathcal C)$ on $L^2(D)$.}
Denote by $c : D \times D \rightarrow \R$ the covariance kernel corresponding to the covariance operator $\mathcal C$. It follows from Mercer's Theorem that the eigenvalues $\{\gamma_j^2\}_{j=1}^\infty$ are positive and summable, and the equality $c(x,y) = \sum_{j=1}^\infty \gamma_j^2 \phi_j(x) \phi_j(y)$ holds for almost all $x,y \in D$.

We have the following result on the spatial regularity of the function $u$ from \cite{charrier12,ds14}.

\begin{lemma}\label{lem:gauss_holder} Let $\mathcal C$ denote the covariance operator with covariance kernel $c$ satisfying {$c(x,y) = g(\|x-y\|)$, for all $x, y \in D$, some norm $\| \cdot\|$ on $\mathbb R^d$ and some Lipschitz continuous function $g \in C^{0,1}(\overline D)$}.
Let $\{\phi_j\}_{j=1}^\infty$ and $\{\gamma_j^2\}_{j=1}^\infty$ be the eigenfunctions and eigenvalues of $\mathcal C$, respectively, and suppose $m_0 \in C^t(\overline D)$, for some $t < 1/2$. Then, the function $u(\cdot ; \xi)$ defined in \eqref{eq:def_upara} is also in $C^t(\overline D)$ almost surely.
\end{lemma}

An example of a covariance kernel $c(x,y)$ that satisfies the assumptions of Lemma \ref{lem:gauss_holder} is the {\em exponential} covariance kernel
\begin{equation}\label{eq:cov_exp}
c(x,y) = \sigma^2 \exp[- \|x-y\|_r\, / \, \lambda],
\end{equation}
where the positive parameters $\sigma^2$ and $\lambda$ are known as the variance and correlation length, respectively, and typically $r=1$ or $2$.

It follows from Lemma \ref{lem:gauss_holder} that if the covariance operator $\mathcal C$ is smooth enough, so that the function $g$ is Lipschitz continuous, the function $u(\cdot ; \xi)$ is almost surely continuous. With $X = C(\overline D)$, it hence follows that $\mu_0(X)=1$.

For practical applications, such as the problem described in Section \ref{ssec:invprob_mod}, it is often of interest to construct a function that is strictly positive on $\overline D$. For this reason, we consider the function $a(\cdot; \xi) = \exp[u(\cdot; \xi)]$. Since $u(\cdot ; \xi)$ is almost surely continuous, we can define almost surely the quantities
\begin{equation*}
a_\mathrm{min}(\xi) = \min_{x \in \overline D} a(x; \xi), \qquad \text{and} \qquad a_\mathrm{max}(\xi) = \max_{x \in \overline D} a(x; \xi).
\end{equation*}
We have the following result on the boundedness of the function $a$ \cite{charrier12,cst13}.

\begin{lemma}\label{lem:gauss_exp} Let the assumptions of Lemma \ref{lem:gauss_holder} hold. Then $a(\cdot; \xi) = \exp(u(\cdot; \xi))$ is in $C^t(\overline D)$ almost surely, for any $t < 1/2$. Furthermore, 
\begin{equation*}
0 < a_\mathrm{min}(\xi) \leq a(x ; \xi) \leq a_\mathrm{max}(\xi) < \infty, \qquad \text{for almost all } x \in D \text{ and } \xi \in \R^\infty,
\end{equation*}
and $a_\mathrm{min}^{-1} \in L^r(\R^\infty)$, $a_\mathrm{max} \in L^r(\R^\infty)$ and $a \in L^r(\R^\infty, C^t(\overline D))$, for all $r \in [1, \infty)$. 
\end{lemma}

Other, smoother covariance kernels, such as the {\em Gaussian} kernel
\[
c(x,y) = \sigma^2 \exp[- \|x-y\|_2^2\, / \, \lambda^2]
\] 
or the kernels from the Mat\'ern family, also satisfy the assumptions of Lemmas \ref{lem:gauss_holder}  and \ref{lem:gauss_exp}, but they lead to a significantly higher spatial regularity $t \ge 1/2$ of $a$.

\subsubsection{Finite-dimensional approximation}

{ In simulations, it is often necessary to use a finite-dimensional approximation of the unknown $u$. Given the parametrisation \eqref{eq:def_upara}, this can be achieved by simply truncating the series at finite truncation order $J$, or through best $N$-term approximations \cite{cds10,twz15}. For simplicity, we here choose the former, and make the following assumption in the remainder of this paper.

\noindent
{\bf Assumption A1.}
({\em Finite Truncation Order}) Suppose the coefficients $\{\gamma_j\}_{j=J+1}^\infty$ are all equal to zero, for some finite $J \in \mathbb N$, such that
\begin{equation}\label{eq:def_upara_fin}
u(x; \xi) = u(x; \xi_{\underline J}) = m_0(x) + \sum_{j=1}^J \gamma_j \xi_j \phi_j(x),
\end{equation}
where $\xi_{\underline J} := \{\xi_j\}_{j=1}^J \in \R^J$.
\vspace{1.5ex}

An important question is how one should optimally choose the truncation order $J$ in \eqref{eq:def_upara_fin}, and the answer typically involves a trade-off between choosing $J$ sufficiently large to retain a required accuracy, and sufficiently small to avoid an unnecessarily large computational cost associated to sampling from $u$.
We will in this paper assume that $J \in \mathbb N$ is given, and will not explicitly discuss how to choose $J$. We refer the interested reader to the works \cite{charrier12,kss12,tsgu13,gknsss15}.

The series \eqref{eq:def_upara_fin} defines a linear mapping $P : \mathbb R^J \rightarrow X$, with $P(\xi_{\underline J}) = u$, and we will define the prior measure $\mu_0$ on $X$ as the pushforward under $P$ of a suitable measure $\mathbb P$ defined on the coefficient space $\mathbb R^J$, equipped with the Borel product $\sigma$-algebra. For the uniform priors considered in section \ref{ssec:invprob_para_unif}, the measure $\mathbb P$ is the product measure
\begin{equation}\label{eq:prodmeas_unif}
\mathbb P(\mathrm d \xi_{\underline J}) = \prod_{j=1}^J \frac{\mathrm d \xi_j}{2}. 
\end{equation}
For the Gaussian priors in section \ref{ssec:invprob_para_gauss}, we have
\begin{equation}\label{eq:prodmeas_gauss}
\mathbb P(\mathrm d \xi_{\underline J}) = \prod_{j =1}^J \frac{1}{\sqrt{2\pi}} \, \exp[-\xi_j^2/2] \, \mathrm d \xi_j. 
\end{equation}


\begin{corollary}  Suppose Assumption A1 holds. Then $(i)$ Lemmas \ref{lem:unif_bound} and \ref{lem:unif_holder} hold; and $(ii)$ if $\{\phi_j\}_{j=1}^J$ and $\{\gamma_j^2\}_{j=1}^J$ are chosen as the first $J$ elements of $\{\phi_j\}_{j=1}^\infty$ and $\{\gamma_j^2\}_{j=1}^\infty$ in Lemma \ref{lem:gauss_holder}, respectively, then Lemma \ref{lem:gauss_holder} and \ref{lem:gauss_exp} hold. 
\end{corollary}
\begin{proof}  Part $(i)$ follows directly from Lemmas \ref{lem:unif_bound} and \ref{lem:unif_holder}, since the parametrisation \eqref{eq:def_upara_fin} is just a special case of \eqref{eq:def_upara}. Part $(ii)$ is proven in \cite{charrier12}.
\end{proof}

\begin{remark}\label{rem:ce} \em{(Alternative Approximations)} The truncated parametrisation \eqref{eq:def_upara_fin} is not the only way to obtain an approximation of $u$ from which we can easily produce samples for simulation. In the case of Gaussian priors, knowledge of the covariance kernel $c$ allows us to assemble the covariance matrix of the Gaussian vector $[u(x_1), u(x_2), \dots, u(x_n)]$, for any $n \in \mathbb N$ and $\{x_i\}_{i=1}^n \subseteq D$, and we can hence use methods based on factorisations of the covariance matrix, such as \cite{dn97}, to sample from $u$ at a finite number of locations in the domain $D$. In applications such as the elliptic problem discussed in section \ref{ssec:invprob_mod}, this is usually sufficient, since typically quadrature methods are used to compute the numerical approximation discussed in section \ref{sec:fe}. For more details, we refer the interested reader to \cite{gknss11,gknss15,teckentrup_thesis}.
\end{remark}
}

\subsection{Model Elliptic Problem}\label{ssec:invprob_mod}

We consider the model inverse problem of determining the distribution of the diffusion coefficient of a {divergence form} elliptic partial differential equation (PDE) from observations of a finite set of noisy continuous functionals of the solution. Let $D \subset \R^d$, for $d = 1,2$ or $3$, be a bounded Lipschitz domain, and denote by $\partial D$ its boundary. 
{The forward problem which underlies the inverse problem of interest here
is to find} the solution $p(\cdot; \xi_{\underline J})$ of the following linear elliptic PDE,\vspace{-1ex}
\begin{equation}\label{eq:model}
\qquad\qquad -\nabla \cdot (k(x; \xi_{\underline J}) \nabla p(x; \xi_{\underline J})) = f(x) \quad \text{in } D, \quad  p(\cdot; \xi_{\underline J}) = 0 \quad \text{on } \partial D,
\end{equation}
for given functions $k(\cdot; \xi_{\underline J}) \in C(\overline D)$ and $f \in H^{-1}(D)$. { Although all results in this section apply also in the case of infinite-dimensional parameter vectors $\xi$, we restrict our attention to finite-dimensional $\xi_{\underline J}$ for consistency.}

The variational formulation of \eqref{eq:model} is to find $p(\cdot ; \xi_{\underline J}) \in H^1_0(D)$ such that \vspace{-1ex}
\begin{equation}\label{def:weak}
\qquad b(p,q; \xi_{\underline J}) = L(q), \qquad \text{for all} \; q \in H^1_0(D),
\end{equation}
where the bilinear form $b$ and the linear functional $L$ are defined as usual, for all $v,w \in H^1_0(D)$ by 
\begin{equation}\label{def:forms}
b(v,w; \xi_{\underline J}) = \int_D k(x; \xi_{\underline J}) \nabla v(x) \cdot \nabla w(x) \dx \quad \text{and} \quad L(w) = \langle f, w \rangle_{H^{-1}(D) , H^1_0(D)}. 
\end{equation}
We say that $p(\cdot; \xi_{\underline J})$ is a weak solution to \eqref{eq:model} iff $p(\cdot ; \xi_{\underline J}) \in H^1_0(D)$ and $p(\cdot ; \xi_{\underline J})$ satisfies \eqref{def:weak}.

In the inverse problem, we  take the coefficient $k$ to be a function of the unknown $u$, in which case both the coefficient $k$ and the solution $p$ depend on the random sequence $\xi_{\underline J}$. When the dependence on $\xi_{\underline J}$ of $k$ and $p$ is irrelevant, we will simply write $k = k(x)$ and $p = p(x)$.
With the unknown function $u(\cdot ; \xi_{\underline J})$ as in Section \ref{ssec:invprob_para}, we choose 
\begin{itemize}
\item $k(\cdot ; \xi_{\underline J})=u(\cdot ; \xi_{\underline J})$, in the case of the uniform priors described in Section \ref{ssec:invprob_para_unif}, and
\item $k (\cdot ; \xi_{\underline J})= k^* + \exp(u(\cdot ; \xi_{\underline J}))$, in the case of Gaussian priors, for some given continuous non-negative function $k^* \ge 0$.
\end{itemize}
By Lemmas \ref{lem:unif_bound} and \ref{lem:gauss_exp}, both these choices ensure that the diffusion coefficient $k(\cdot;\xi_{\underline J})$ in \eqref{def:weak} is strictly positive on $D$, $\mathbb P$-almost surely.
In terms of the notation in previous sections, the Banach space $X$ is the space of continuous functions $C(\overline D)$ as before. The forward map $\mathcal{G}$ is defined by $\mathcal{G}(u) = p$, i.e. it maps the unknown function $u(\cdot ; \xi_{\underline J})$ to the solution $p(\cdot ; \xi_{\underline J})$. (Note that the definition of $\mathcal{G}$ differs between the two choices $k=u$ and $k=k^* + \exp(u)$.) We take the Banach space $V$ as the Sobolev space $H^1_0(D)$.

Existence and uniqueness of the weak solution $p(\cdot ; \xi_{\underline J})$ is ensured by the Lax-Milgram Theorem. As in previous sections, let $k_\mathrm{min}(\xi_{\underline J})$ and $k_\mathrm{max}(\xi_{\underline J})$ be such that
\begin{equation*}
0 < k_\mathrm{min}(\xi_{\underline J}) \leq k(x ; \xi_{\underline J}) \leq k_\mathrm{max}(\xi_{\underline J}) < \infty, \qquad \text{for almost all } x \in D \text{ and for } \xi_{\underline J} \ \mathbb{P}\text{-almost surely}.
\end{equation*}
For uniform priors, $k_\mathrm{min}(\xi_{\underline J})$ and $k_\mathrm{max}(\xi_{\underline J})$ are independent of $\xi_{\underline J}$. If $k^*(x) > 0$, for all $x \in \overline{D}$, then $k_\mathrm{min}(\xi_{\underline J})$ is also independent of $\xi_{\underline J}$ in the Gaussian case. 
\begin{definition}\em 
We will refer to the coefficient $k$ as {\em uniformly elliptic} (respectively {\em uniformly bounded}) when $k_\mathrm{min}(\xi_{\underline J})$ (respectively $k_\mathrm{max}(\xi_{\underline J})$) is independent of $\xi_{\underline J}$.
\end{definition}

The following is a direct consequence of the Lax-Milgram Lemma and Lemmas \ref{lem:unif_bound} and \ref{lem:gauss_exp}.

\begin{lemma} \label{lem:laxmil} For $\mathbb P$-almost all $\xi_{\underline J} \in \R^J$, there exists a unique weak solution $p(\cdot; \xi_{\underline J}) \in H^1_0(D)$ to the variational problem \eqref{def:weak} and
\begin{equation*}
\left|p(\cdot; \xi_{\underline J})\right|_{H^1(D)} \leq \frac{\|f\|_{H^{-1}(D)}}{k_\mathrm{min}(\xi_{\underline J})}\,.
\end{equation*}
Furthermore, $p \in L^r_{\mathbb P}(\R^J, H^1_0(D))$, for all $r \in [1, \infty)$. If $k$ is uniformly elliptic, then the result holds also for $r=\infty$.
\end{lemma}

In order to conclude on the well-posedness of the posterior distribution $\mu^y$, we furthermore have the following result on the continuity of the forward map $\mathcal{G}$.

\begin{lemma}\label{lem:for_con} The map $\mathcal{G} : X \rightarrow V$, $\mathcal{G}(u) = p$, is continuous.
\end{lemma}
\begin{proof}
Denote by $p_1$ and $p_2$ two weak solutions of \eqref{def:weak} with the same right hand side $f$ and with coefficients $k_1$ and $k_2$, respectively. Let $k_\mathrm{min}$ and $k_\mathrm{max}$ be such that
\begin{equation*}
0 < k_\mathrm{min} \leq k_i(x) \leq k_\mathrm{max} < \infty, \quad \text{for almost all } \; x \in D,
\end{equation*}
for $i=1,2$. Then it follows from the variational formulation \eqref{def:weak} that
\begin{equation*}
|p_1 - p_2|_{H^1(D)} \leq \frac{\|f\|_{H^{-1}(D)}}{k_\mathrm{min}^2} \; \|k_1 - k_2 \|_{C(\overline D)}.
\end{equation*}
In the case $k=u$, the continuity of $\mathcal{G}$ now follows immediately. In the case $k=\exp(u),$ the continuity of $\mathcal{G}$ follows from the continuity of the exponential function.
\end{proof}

We then have the following corollary to {Proposition} \ref{thm:rad_nik}, which follows immediately from Lemmas \ref{lem:unif_bound}, \ref{lem:gauss_holder} and \ref{lem:for_con}, together with the continuity of the observation operator $\mathcal H$.

\begin{corollary} For the forward map $\mathcal{G}$ defined by $\mathcal{G}(u) = p$, the posterior measure $\mu^y$ is absolutely continuous with respect to the prior measure $\mu_0$, with Radon-Nikodym derivative \eqref{eq:rad_nik}.
\end{corollary}

\begin{remark}\label{rem:piecewise} \em (Piecewise continuous coefficients) Although we here restrict our attention to the case of continuous random coefficients, the theory extends to the piecewise continuous case where a further source of randomness can be introduced in the partitioning of the computational domain $D$ into sub-domains. The well-posedness of the posterior distribution in this case was shown in \cite{ds15}. The regularity and spatial discretisation error (as discussed in Section \ref{sec:fe}) were analysed in \cite{tsgu13,teckentrup_thesis}.
\end{remark}

\section{Finite Element Discretisation}\label{sec:fe}

In this section, we analyse the error introduced in the computation of the prior expectations $Z$ and $Q$ by a finite element approximation of the forward map $\mathcal{G}$. We consider only standard, continuous, piecewise linear finite elements on polygonal/polyhedral domains in detail. To this end, denote by $\{\mathcal T_h\}_{h>0}$ a shape-regular family of simplicial triangulations of the Lipschitz polygonal/polyhedral domain $D$, parametrised by their mesh width $h := \max_{\tau \in \mathcal T_h} \text{diam}(\tau)$. Associated with each triangulation $\mathcal T_h$ we define the space
\begin{equation}\label{def:fe_space}
V_h := \left\{ q_h \in C(\overline D) : q_h |_\tau \; \text{linear for all } \, \tau \in \mathcal T_h \; \; \text{and} \; \; q_h |_{\partial D} = 0 \right\}
\end{equation} 
of continuous, piecewise linear functions on $D$ that vanish on the boundary $\partial D$.

The finite element approximation to \eqref{def:weak}, denoted by $p_h$, is now the unique function in $V_h$ that satisfies 
\begin{equation}\label{def:fe}
b(p_h,q_h; \xi_{\underline J}) = L(q_h), \qquad \text{for all} \; q_h \in V_h,
\end{equation}
where the bilinear form $b$ and the functional $L$ are as in \eqref{def:forms}. Note that, in particular, this implies that $p_h$ satisfies the same bound as in Lemma \ref{lem:laxmil}: 
\begin{equation}\label{eq:fe_h1}
|p_h(\cdot ; \xi_{\underline J})|_{H^1(D)} \leq \|f\|_{H^{-1}(D)} / k_\mathrm{min}(\xi_{\underline J})\,.
\end{equation}
The approximate forward map $\mathcal{G}_h : X \rightarrow V$ is then defined by $\mathcal{G}_h(u) = p_h$, and we denote the resulting approximations of $Z$ and $Q$, respectively, by
\[
Z_h = \EE_{\mu_0}[\theta(\mathcal{G}_h(u))] \quad  \text{and} \quad Q_h = \EE_{\mu_0}[\psi(\mathcal{G}_h(u))].
\]

A standard technique to prove convergence of finite element approximations of functionals is to use a duality argument, similar to the classic Aubin Nitsche trick used to prove optimal convergence rates for the $L^2$ norm. In the context of the elliptic PDE \eqref{eq:model} with random coefficients, this analysis was performed in \cite{tsgu13}. We here summarise the main results of the error analysis, and show that if the observation operator $\mathcal H$ and the functional of interest $\phi$ are smooth enough, the finite element error in the prior expectations $Z_h$ and $Q_h$ converges at the optimal rate.

Let $v, w \in H^1_0(D)$. Given a functional $F : H^1_0(D) \rightarrow \mathbb R$, we denote by $D_v F(w)$ its Fr\'echet derivative at $w$, { applied to} $v$. With $p$ and $p_h$ as before, we define 
\begin{align*}
\overline{D_v F} (p,p_h) &= \int_0^1 D_v F (p + \lambda(p_h - p)) \dlambda, \\
\text{and} \quad \overline{|D_v F|} (p,p_h) &= \int_0^1 |D_v F (p + \lambda(p_h - p))| \dlambda,
\end{align*}
which in some sense are averaged derivatives of $F$ on the path from $p$ to $p_h$. Let us now define the following dual problem: find $z \in H^1_0(D)$ such that
\begin{equation}\label{def:weak_dual}
b(q,z; \xi_{\underline J}) = \overline{D_q F} (p,p_h), \qquad \text{for all} \; q \in H^1_0(D).
\end{equation}
Denote the finite element approximation of the dual solution $z$ by $z_h \in V_h$. It then follows from the Fundamental Theorem of Calculus, Galerkin orthogonality of the primal problem \eqref{def:weak} and boundedness of the bilinear form $b$ that
\begin{equation*}
|F(p) - F(p_h)| = |b(p-p_h, z ; \xi_{\underline J})| = |b(p-p_h, z - z_h ; \xi_{\underline J})| \leq k_\mathrm{max}(\xi_{\underline J}) \; |p - p_h|_{H^1(D)} \; |z - z_h|_{H^1(D)}.
\end{equation*}
In order to prove convergence of the finite element error $|F(p) - F(p_h)| $, it hence suffices to prove convergence of $|p - p_h|_{H^1(D)}$ and $|z - z_h|_{H^1(D)}$. For our further analysis, we make the following assumption on the smoothness of the maps $\phi$ and $\mathcal H$.
{Examples of functionals satisfying Assumption A2  are discussed in \cite{tsgu13}, and include linear functionals, powers of linear functionals and boundary fluxes.}

\noindent
{\bf Assumption A2.}
({\em Differentiability}) Let $\phi$ and $\mathcal H_i$, $i=1, \dots, m$, be continuously Fr\'echet differentiable on the path $\{p + \lambda (p - p_h)\}_{\lambda \in [0,1]}$, and suppose that there exist $t_* \in [0,1], q_* \in [1,\infty]$ and $C_\phi, C_{\mathcal H} \in L^{q_*}_{\mathbb P}(\mathbb R^J)$ such that $f \in H^{t_*-1}(D)$,
\[
\overline{|D_v \phi|}(p,p_h) \leq C_\phi(\xi_{\underline J}) \|v\|_{H^{1-t_*}(D)}, \quad \text{ and } \quad \overline{|D_v \mathcal H_i|}(p,p_h) \leq C_{\mathcal H}(\xi_{\underline J}) \|v\|_{H^{1-t_*}(D)},
\]
for all $v \in H^1_0(D)$ and almost all $\xi_{\underline J} \in \R^J$. 
\vspace{1.5ex}

Let now $F = \phi$ or $F=\mathcal H_i$, for some $i \in \{1, \dots, m\}$. To get well-posedness of the primal problem \eqref{def:weak} and the dual problem \eqref{def:weak_dual}, as well as existence and uniqueness of the solutions $p(\cdot ; \xi_{\underline J}) \in H^1_0(D)$ and $z(\cdot ; \xi_{\underline J}) \in H^1_0(D)$, for almost all $\xi_{\underline J} \in \R^J$, it is sufficient to assume that Assumption A2 holds with $t_*=0$. However, in order to prove convergence of the finite element approximations, it is necessary to require stronger spatial regularity of $p$ and $z$, which requires Assumption A2 to hold for some $t_* > 0$. We have the following result from \cite{cst13,tsgu13}. The assumptions on $D$ being polygonal and convex are purely to simplify the presentation. {Proposition \ref{thm:spat_reg}} also holds for piecewise smooth or for non-convex domains, but typically with stronger restrictions on the range of $s$.

\begin{proposition}\label{thm:spat_reg} Let $D$ be a Lipschitz polygonal, convex domain, let $k \in L^{r_*}_\mathbb{P}(\R^J, C^t(\overline D))$ and $k_\mathrm{min}^{-1} \in L^{r_*}_\mathbb{P}(\R^J)$, for some $t \in (0,1]$ and $r_* \in [1, \infty]$, and let Assumption A2 hold with $t_* = t$ and $q_* = r_*$. Then, the solutions $p$ and $z$ of \eqref{def:weak} and \eqref{def:weak_dual} are both in $L^r_\mathbb{P}(\R^J, H^{1+s}(D))$, for any $s < t$ and $r < r_*$. The result also holds for $r = r_*$ or for $s=t$, if $r_* = \infty$ or $t=1$, respectively.
\end{proposition}

We will now show that the functionals $\theta$ and $\psi$ appearing in the prior expectations $Z$ and $Q$, respectively, satisfy the bounds in Assumption A2 provided $\phi$ and $\mathcal H$ satisfy Assumption A2, as well as the growth conditions in Assumption A3 below. \vspace{1.5ex}

\noindent
{\bf Assumption A3.} ({\em Boundedness})
Suppose there are constants $M_1,M_2 >0$ and $n_1,n_2 \in \mathbb{N}$, such that
\begin{equation}
\label{polygrowth}
|\phi(v)| \le M_1 \left(1+|v|^{n_1}_{H^1(D)}\right) \quad \text{and} \quad |\mathcal H(v)| \le M_2\left(1+|v|^{n_2}_{H^1(D)}\right), \quad \text{for all} \ \ v \in H^1_0(D).
\end{equation}

Recall the product and chain rules for Fr\'echet derivatives for functionals $F_1, F_2 : H^1_0(D) \rightarrow \mathbb R$ and a function $f : \mathbb R \rightarrow \mathbb R$:
\[
D_v (F_1 F_2) (w) = F_2 (w) D_v F_1(w) + F_1 (w) D_v F_2(w) \quad \text{and} \quad D_v (f \circ F_1)(w) = D_{D_v F_1(w)} f(F_1(w)).
\]

We then have the following result.

\begin{lemma}\label{lem:func_bound} Let Assumption A2 hold with $t_* \in [0,1]$ and $q_* \in [1,\infty]$, and suppose Assumption A3 holds. Then
\[
\overline{| D_v \theta |}(p,p_h) \leq C_\theta(\xi_{\underline J}) \|v\|_{H^{1-t_*}(D)}, \quad \text{ and } \quad \overline{| D_v \psi  |}(p,p_h) \leq C_{\psi}(\xi_{\underline J}) \|v\|_{H^{1-t_*}(D)},
\]
for all $v \in H^1_0(D)$ and for $\xi_{\underline J}$ $\mathbb P$-almost surely, where $C_\theta(\xi_{\underline J})$ and $C_{\psi}(\xi_{\underline J})$ are in $L^{r}_\mathbb{P}(\R^J)$, for all $r \in [1,q_*)$ If $k$ is uniformly elliptic, then the result holds also for $r = q_*$.
\end{lemma}
\begin{proof} First, we use the chain rule {and product rule} for Fr\'echet derivatives to obtain
\[
D_v \theta(w) = D_v \left(\exp\left[-\frac{1}{2 \sigma_\eta^2} \sum_{i=1}^m (y_i - \mathcal H_i(w))^2\right]\right)
= - \theta(w) \, \frac{1}{\sigma_\eta^2} \, \sum_{i=1}^m \, 
D_v \mathcal H_i(w)(y_i - \mathcal H_i(w)).
\]
Denoting $p_\lambda = p + \lambda(p_h-p)$, we then have
\[
\overline{|D_v \theta|}(p,p_h) =  \int_0^1 \left| D_v \theta(p_\lambda) \right| \dlambda \leq \int_0^1 \theta(p_\lambda) \, \, \frac{1}{\sigma_\eta^2} \sum_{i=1}^m \left| D_v \mathcal H_i(p_\lambda)\right| |y_i - \mathcal H_i(p_\lambda))| \dlambda.
\]
By the definition of $\theta$ in \eqref{eq:def_theta}, we have $\theta(p_\lambda) \leq 1$, for all $\lambda \in [0,1]$. By Assumption A3, it follows that
\[
|y_i - \mathcal H_i(p_\lambda))| \leq C (1+ |p_\lambda|^{n_2}_{H^1(D)}) \leq C ( 1+ \|f\|^{n_2}_{H^{-1}(D)} k^{-n_2}_\mathrm{min}(\xi_{\underline J})),
\]
for some (generic) constant $C$ independent of the mesh size $h$ and random parameter $\xi_{\underline J}$. It then follows that
\[
\overline{|D_v \theta|}(p,p_h) \leq \frac{C}{\sigma_\eta^2}  ( 1+ \|f\|^{n_2}_{H^{-1}(D)} k^{-n_2}_\mathrm{min}(\xi_{\underline J})) \, \sum_{i=1}^m \overline{|D_v \mathcal H_i |}(p,p_h).
\]
With $C_\theta(\xi_{\underline J}) = \frac{m \, C}{\sigma_\eta^2} \, ( 1+ \|f\|^{n_2}_{H^{-1}(D)} k^{-n_2}_\mathrm{min}(\xi_{\underline J})) \, C_\mathcal{H}(\xi_{\underline J})$, it then follows from Assumption A2 that
\[
\overline{|D_v \theta|}(p,p_h) \leq C_\theta(\xi_{\underline J}) \|v\|_{H^{1-t_*}(D)}.
\]

Next, using the product rule for Fr\'echet derivatives, together with the result just proved, we have 
\[
D_v \psi(w) = \theta(w) \, D_v \phi(w) - \phi(w) \, \theta(w) \, \frac{1}{\sigma_\eta^2} \,\sum_{i=1}^m  D_v \mathcal H_i(w) (\delta_i - \mathcal H_i(w)). 
\]
With $p_\lambda $ as before, it then follows that
\begin{align*}
\overline{|D_v \psi|}(p,p_h) &=  \int_0^1 \left| D_v \psi(p_\lambda) \right| \dlambda  \\
&\leq \int_0^1 \theta(p_\lambda) \, \left[|D_v \phi(p_\lambda)| + \phi(p_\lambda) \, \frac{1}{\sigma_\eta^2}  \sum_{i=1}^m \,\left| D_v \mathcal H_i(p_\lambda)\right| |\delta_i - \mathcal H_i(p_\lambda)| \right]\dlambda. 
\end{align*}
Now $\theta(p_\lambda) \leq 1$, for all $\lambda \in [0,1]$. By Assumption A3, it follows that
\[
|\phi(p_\lambda)| \; \max_{i \in \{1, \dots, m\}} |y_i - \mathcal H_i(p_\lambda))| \leq C (1+ |p_\lambda|^{n_2}_{H^1(D)})^2 \leq C ( 1+ \|f\|^{2 n_2}_{H^{-1}(D)} k^{- 2 n_2}_\mathrm{min}(\xi_{\underline J})),
\]
for some (generic) constant $C$ independent of the mesh size $h$ and random parameter $\xi_{\underline J}$. It then follows that  
\[
\overline{|D_v \psi|}(p,p_h) \leq \overline{|D_v \phi|}(p,p_h) + \frac{C}{\sigma_\eta^2} ( 1+ \|f\|^{2 n_2}_{H^{-1}(D)} k^{- 2 n_2}_\mathrm{min}(\xi_{\underline J})) \, \sum_{i=1}^m \overline{|D_v \mathcal H_i|}(p,p_h).
\]
With $C_{\psi}(\xi_{\underline J}) = C_{\phi}(\xi_{\underline J}) + \frac{m \, C}{\sigma_\eta^2} ( 1+ \|f\|^{2 n_2}_{H^{-1}(D)} k^{- 2 n_2}_\mathrm{min}(\xi_{\underline J})) \, C_\mathcal{H}(\xi_{\underline J})$, it then follows that
\[
\overline{|D_v \psi|}(p,p_h) \leq C_\psi(\xi_{\underline J}) \|v\|_{H^{1-t_*}(D)}.
\]

{Finally, recall that by Lemmas \ref{lem:unif_bound} and \ref{lem:gauss_exp}, we have $k^{-1}_\mathrm{min}(\xi_{\underline J}) \in L_\mathbb{P}^{\infty}(\R^J)$ if $k$ is uniformly elliptic, and $k^{-1}_\mathrm{min}(\xi_{\underline J}) \in L_\mathbb{P}^{q}(\R^J)$, for any $1 \leq q < \infty$, otherwise.} Hence, it follows from Assumption A2,  together with H\"older's and Minkowski's inequalities, that $C_\theta(\xi_{\underline J})$ and $C_\psi(\xi_{\underline J})$ are in $L^{r}_\mathbb{P}(\R^J)$, for all $r \in [1,q_*)$. If $k$ is uniformly elliptic, we can also set $r = q_*$.
\end{proof}

Bounds on the finite element errors $|\theta(p) - \theta(p_h)|$ and $|\psi(p) - \psi(p_h)|$ 
now follow directly {from Proposition} \ref{thm:spat_reg} and Lemma \ref{lem:func_bound}.

\begin{theorem}\label{thm:fe_post} Under the assumptions of {Proposition} \ref{thm:spat_reg} and Lemma \ref{lem:func_bound} with $t \in (0,1]$ and $r_* \in [1, \infty]$, we have
\[
\|\theta(p) - \theta(p_h)\|_{L^r_\mathbb{P}(\R^J)} \leq C_{k,f,\theta,D} \; h^{2s}, \quad \text{and} \quad \|\psi(p) - \psi(p_h)\|_{L^r_\mathbb{P}(\R^J)} \leq C_{k,f,\psi,D} \; h^{2s},
\]
for any $s<t$ and $r<r_*$. The constants $C_{k,f,\theta,D}$ and $C_{k,f,\psi,D}$ are independent of $h$. If $r_* = \infty$, we can also bound the $L_\mathbb{P}^\infty$ norms. If $t=1$, we can set $s=1$.
\end{theorem}

{ Proposition \ref{thm:spat_reg}, Lemma \ref{lem:func_bound} and Theorem \ref{thm:fe_post} hold, without any additional assumptions, also for infinite-dimensional parameter vectors $\xi \in \R^\infty$ \cite{cst13,tsgu13}.}

\section{Sampling methods}\label{sec:samp}

In this section, we briefly recall the main ideas behind Monte Carlo (MC), Multilevel Monte Carlo (MLMC) and quasi-Monte Carlo (QMC) estimators to compute the prior expectation 
$Q_h = \EE_{\mu_0}[\psi(p_h)]$. The estimators for $Z_h = \EE_{\mu_0}[\theta(p_h)]$ are defined analogously. We also provide bounds on the sampling error of the estimators, which will become useful for bounding the mean square error in Section \ref{sec:mse}. For more details, we refer the reader to \cite{robert_casella,cgst11,kss12,gknsss15}.

\subsection{Monte Carlo estimators}

The standard Monte Carlo estimator for $Q_h$ is\vspace{-0.5ex}
\begin{equation}\label{eq:mc_est}
\qquad \widehat Q_{h,N}^\mathrm{MC} = \frac{1}{N} \sum_{i=1}^N \psi(p_h(\cdot ; \xi_{\underline J}^{(i)})),
\end{equation}
where $\xi_{\underline J}^{(i)}$ is the $i$th sample of $\xi_{\underline J}$ from the distribution $\mathbb P$, and $N$ independent samples are computed in total. The estimator \eqref{eq:mc_est} is an unbiased estimator of $Q_h$, with variance 
\begin{equation}\label{eq:mc_var}
\VV[\widehat Q_{h,N}^{MC}] = \frac{\VV[\psi(p_h)]}{N}.
\end{equation}

\subsection{Multilevel Monte Carlo estimators}

The main idea of multilevel Monte Carlo estimation is simple. Linearity of the expectation operator implies that\vspace{-1ex} 
\[
\EE_{\mu_0}[\psi(p_h)] = \EE_{\mu_0}[\psi(p_{h_0})] + \sum_{\ell=1}^L \EE_{\mu_0}[\psi(p_{h_\ell}) - \psi(p_{h_{\ell-1}})],
\]
where $\{h_\ell\}_{\ell=0}^L$ are the mesh widths of a sequence of increasingly fine triangulations $\mathcal T_{h_\ell}$ with $h_L = h$, the finest mesh width, and $k_1 \leq h_{\ell-1}/h_\ell \leq k_2$, for all $\ell=1, \dots, L$ and some $1 < k_1 \leq k_2 < \infty$. The multilevel idea is now to estimate each of the terms independently using a Monte Carlo estimator.
Setting for convenience $Y_0^\psi = \psi(p_{h_0})$, and $Y_\ell^\psi = \psi(p_{h_\ell}) - \psi(p_{h_{\ell-1}})$, for $\ell=1,\dots, L$, we define the MLMC estimator as
\begin{equation}\label{eq:mlmc_est}
\widehat Q_{h,\{N_{\ell}\}}^\mathrm{ML} = \sum_{\ell=0}^L \widehat Y_{\ell, N_\ell}^{\psi,\mathrm{MC}} = \sum_{\ell=0}^L \frac{1}{N_\ell} \sum_{i=1}^{N_\ell} Y_\ell^\psi(\cdot ; \xi_{\underline J}^{(i, \ell)}),
\end{equation}
where importantly the quantity $Y_\ell^\psi(\cdot ; \xi_{\underline J}^{(i,\ell)})$ uses the same sample $\xi_{\underline J}^{(i,\ell)}$ on both meshes. 
The estimator \eqref{eq:mlmc_est} is an unbiased estimator of $Q_h$, with variance
\begin{equation}\label{eq:mlmc_var}
\VV[\widehat Q_{h,\{N_{\ell}\}}^\mathrm{ML}] = \sum_{\ell=0}^L \frac{\VV[Y_\ell^\psi]}{N_\ell} \leq C \sum_{\ell=0}^L \frac{h_\ell^{4s}}{N_\ell},
\end{equation}
where the last inequality follows from Theorem \ref{thm:fe_post}, with a constant $C$ independent of $\{h_\ell\}_{\ell=0}^L$ and with $0 \le s < t \le 1$, as defined in {Proposition} \ref{thm:spat_reg}.

{When defining the MLMC estimator \eqref{eq:mlmc_est}, one can in fact also use level-dependent truncation levels $J_\ell$. This approach was analysed in \cite{tsgu13}, and can lead to further significant gains in terms of computational cost.}

\subsection{Quasi-Monte Carlo estimators}\label{ssec:samp_qmc}
{Quasi-Monte Carlo methods are classically formulated as quadrature rules over the unit cube $[0,1]^J$, for some $J \in \mathbb N$.}
Treating $\xi_{\underline J}$ as a deterministic parameter vector distributed according to the product uniform or Gaussian measure, respectively,
\begin{equation}\label{eq:int_qmc}
\EE_{\mu_0}[\psi(p_h)] = \int_{[0,1]^J} \psi(p_h(\cdot; (\Phi^{-1}_J(v)))) \mathrm{d}v,
\end{equation}
where $\xi_{\underline J} = \Phi_J^{-1}(v)$ denotes the inverse cumulative normal applied to each entry of $v$ in the Gaussian case. In the uniform case, $\Phi_J^{-1}$ is the simple change of variables mapping $v_j$ to $2v_j-1$. We will use a randomly shifted lattice rule to approximate the integral \eqref{eq:int_qmc}. This takes the form 
\begin{equation}\label{eq:qmc_est}
\widehat Q_{h,N}^\mathrm{QMC} = \frac{1}{N} \sum_{i=1}^N \psi\Big(p_h\big(\cdot ; \tilde \xi_{\underline J}^{(i)}\big)\Big), \quad \text{where} \ \ \tilde \xi_{\underline J}^{(i)} := \Phi^{-1}_J\bigg(\text{frac}\Big(\frac{i z}{N} + \Delta\Big)\bigg),
\end{equation}
$z \in \{1,\ldots,N-1\}^J$ is a {\em generating vector}, $\Delta$ is a uniformly distributed {\em random shift} on $[0,1]^J$, and "frac" denotes the fractional part function, applied component-wise. To ensure that every one-dimensional projection of the lattice rule has $N$ distinct values we furthermore assume that each component $z_j$ of $z$ satisfies $\text{gcd}(z_j,N) = 1$
(cf.~\cite{dks13}).

The variance of the QMC estimator \eqref{eq:qmc_est} is given by
\begin{equation}\label{eq:qmc_var}
\VV[\widehat Q_{h,N}^\mathrm{QMC}] = \EE_\Delta [ (\EE_{\mu_0}[\psi(p_h)] - \widehat Q_{h,N}^\mathrm{QMC} )^2 ].
\end{equation}
To bound it, we make the following assumption on the integrand $\psi(p_h)$.

\noindent
{\bf Assumption A4.} Let $c_1 > 0$ be a constant independent of $J$ and let $b_j := \gamma_j \|\phi_j\|_{C^0(\overline D)}$, for $j\in \mathbb{N}$. We assume that, for any multi-index $\boldsymbol{\nu} \in \{0,1\}^J$ with $|\boldsymbol{\nu}| = \sum_{j\le J} \nu_j$\,,
\[
\left| \frac{\partial^{|\boldsymbol{\nu}|}\psi(p_h)}{\partial \xi_{\underline J}^{\boldsymbol{\nu}}} \right| \le C_{k,f,\psi,D} \frac{c_1^{|\boldsymbol{\nu}|} |\boldsymbol{\nu}|!}{k_\mathrm{min}(\xi_{\underline J})} \prod_{j = 1}^{J} b_j^{\nu_j}\,.\vspace{1ex}
\]

For linear functionals $\psi$ on $H^1_0(D)$, this has been proved in \cite{kss12} and in \cite{gknsss15} for the uniform and the Gaussian cases, respectively.
In both cases, we can choose $c_1 = 1/\ln 2$.
However, in the Bayesian setting, both $\psi$ and $\theta$ are inherently non-linear functionals of $p$. Nevertheless, if $\phi$ and $\mathcal{H}$ are linear functionals of $p$, and $k$ is uniformly elliptic, Assumption A4 can be proved by using the classical Fa\`{a} di Bruno formula \cite{cs96}, a multidimensional version of the chain rule. A proof for $\theta$ in the case $m=1$ can be found in Appendix \ref{sec:appendix}. We omit the proof for $\psi$ or for $m>1$. A proof for general analytic functionals $\phi$ and $\mathcal{H}$ of $p$ would be even more technical and require the use of generalisations of Fa\`{a} di Bruno's formula to Fr\'echet derivatives.
{For this reason, we simply work under {Assumption A4.}}

{In the case of uniform priors, Assumption A4 was proven to hold in \cite{dggs16} using arguments from complex analysis and the holomorphy of $p_h$ as a function of $\xi_{\underline J }$.} 

\begin{lemma}\label{lem:qmc_var_linear} Suppose Assumption A4 holds and the sequence $\{b_j\}_{j=1}^\infty$ is in $l^q(\R^\infty)$, for some $q \in (0,1]$. Then, a randomly shifted lattice rule can be constructed via a component-by-component algorithm in {$\mathcal{O}(J N \log N)$} cost, such that
\[
\VV[\widehat Q_{h,N}^\mathrm{QMC}] \leq \left\{ 
\begin{array}{ll}
         C_{\psi,q,\delta} \; N^{-1/\delta}, & \mbox{if $q \in (0,2/3]$},\\
        C_{\psi,q}\; N^{-(1/q - 1/2)}, & \mbox{if $q \in (2/3,1)$},\end{array}
\right.
\]
for any $\delta \in (1/2,1]$, independently of $J$. For $q=1$ and under further assumptions given in \cite[Theorem 20]{gknsss15} and \cite[Theorem 6.4]{kss12}, we have $\VV[\widehat Q_{h,N}^\mathrm{QMC}] \leq C_{\psi}\, N^{- 1/2}$.
\end{lemma}
\begin{proof}
The proof follows those in \cite{kss12,gknsss15} with suitable changes to the product and order dependent (POD) weights if $c_1 \not= 1/\ln 2$.
\end{proof}
It is even possible to combine quasi-Monte Carlo sampling and multilevel estimation and the gains are complementary \cite{kss15,kssu15}, but we will not include these estimators or their analysis here.

\section{Mean Square Error and Computational Complexity}\label{sec:mse}

We will now use the results from Sections \ref{sec:fe} and \ref{sec:samp} to bound the mean square error (MSE) of estimators for the ratio $Q/Z$.
To this end, let us denote by $\widehat Z_{h}$ and $\widehat Q_{h}$ one of the Monte Carlo type estimators discussed in Section \ref{sec:samp} for $Z_h$ and $Q_h$, respectively. 
Let us define the mean square error
\begin{equation}\label{eq:rat_mse}
e\left(\widehat Q_{h} \Big/ \widehat Z_{h}\right)^2 = \EE\left[ \bigg( \frac{Q}{Z} - \frac{\widehat Q_{h}}{\widehat Z_{h}}\bigg)^2 \right] .
\end{equation}
For MC and MLMC estimators, the expectation in the expression \eqref{eq:rat_mse} above is with respect to the prior measure $\mu_0$ on $X$. For QMC estimators, it is with respect to the random shift $\Delta$. 

Rearranging the mean square error and using the triangle inequality, we have
\begin{align}
e\left(\widehat Q_{h} \Big/ \widehat Z_{h}\right)^2  
& = \frac{1}{Z^2} \, \EE\left[ \Big( Q - \widehat Q_{h} + (\widehat Q_{h} / \widehat Z_{h})\left(\widehat Z_{h} - Z\right) \Big)^2\right] \nonumber \\[1ex]
\label{eq:rat_mse2}
&\le \frac{2}{Z^2} \bigg(\EE\left[ ( \widehat Q_{h} - Q )^2\right] \, + \, \EE\left[ (\widehat Q_{h} / \widehat Z_{h})^{2} ( \widehat Z_{h} - Z)^2 \right] \bigg). 
\end{align}
Our further analysis depends on the integrability of $\widehat Q_{h} \big/ \widehat Z_{h}$, and we thus consider separately the cases of uniformly and non-uniformly elliptic coefficients $k$.

\subsection{Uniformly elliptic case}\label{ssec:mse_unif}

In the case where the coefficient $k$ is uniformly elliptic, we have the following result on the integrability of $\widehat Q_{h} / \widehat Z_{h}$. The assumptions on $\phi$ and $\mathcal H$ are more general than Assumption A3, and allow for very general non-linear growth. Since, in general, $\widehat Z_{h,\{N_{\ell}\}}^\mathrm{ML}$ could be negative, we require stronger assumptions in the case of multilevel Monte Carlo estimators. { In particular, we require the assumptions of Theorem \ref{thm:fe_post} to hold with $r=\infty$, which means that the coefficient $k$ needs to be uniformly bounded as well as uniformly elliptic. The analysis of MLMC in Lemma \ref{lem:rat_linfty} below therefore does not apply in the case of Gaussian priors.}

\begin{lemma}\label{lem:rat_linfty} Suppose $k$ is uniformly elliptic and there are two constants $M_1,M_2>0$, such that 
\[
|\phi(v)|\le M_1 \quad  \text{and} \quad |\mathcal H(v)|\le M_2, \quad \text{for all} \ \ v \in H^{1}_0(D) \ \ \text{with} \ \ |v|_{H^1(D)} \le \|f\|_{H^{-1}(D)} / k_{\min}.
\]
Then $\widehat Q_{h,N}^\mathrm{MC} \big/ \widehat Z_{h,N}^\mathrm{MC} \in L^\infty_\mathbb{P}(\mathbb R^J)$, and $\widehat Q_{h,N}^\mathrm{QMC} \big/\widehat Z_{h,N}^\mathrm{QMC} \in L^\infty_\Delta([0,1]^J)$, with $L^\infty$-norms bounded independently of $h$ and $N$. 

If in addition $h_0$ is sufficiently small and the assumptions of Theorem \ref{thm:fe_post} hold with $r=\infty$, we also have $\widehat Q_{h,\{N_{\ell}\}}^\mathrm{ML} \big/ \widehat Z_{h,\{N_{\ell}\}}^\mathrm{ML} \in L^\infty_\mathbb{P}(\mathbb R^J)$, with $L^\infty$-norm bounded independently of $\{h_\ell\}$, $\{N_\ell\}$ and $L$.
\end{lemma}
\begin{proof}
Using the definition of $\widehat Q_{h,N}^\mathrm{MC}$ in \eqref{eq:mc_est}, as well as the bound in \eqref{eq:fe_h1} and the fact that $\theta(v) \leq 1$, for all $v \in H^1_0(D)$, it follows that
\[
\qquad |\widehat Q_{h,N}^\mathrm{MC}| = \left| \frac{1}{N} \sum_{i=1}^N \phi(p_h(\cdot ; \xi_{\underline J}^{(i)})) \theta(p_h(\cdot ; \xi_{\underline J}^{(i)})) \right|\leq M_1
\]
and
\[
\widehat Z_{h,N}^\mathrm{MC} = \frac{1}{N} \sum_{i=1}^N \theta(p_h(\cdot ; \xi_{\underline J}^{(i)})) \geq \exp\left(-\frac{|y|^2 + m M_2^2}{2\sigma_\eta^2}\right) =: b > 0\,.
\]
Since the upper bound on $\widehat Q_{h,N}^\mathrm{MC}$ and 
the lower bound on $\widehat Z_{h,N}^\mathrm{MC}$ are independent of the random samples $\{\xi_{\underline J}^{(i)}\}_{i=1}^N$, the claim of the lemma follows for Monte Carlo estimators. The proof for quasi-Monte Carlo estimators is identical.

For multilevel Monte Carlo estimators, an upper bound on $\widehat Q_{h,\{N_\ell\}}^\mathrm{ML}$ follows as before. On the other hand, to bound $\widehat Z_{h,\{N_{\ell}\}}^\mathrm{ML}$ we can use Theorem \ref{thm:fe_post} which implies that
\[
\widehat Z_{h,\{N_{\ell}\}}^\mathrm{ML} \geq b - \sum_{\ell=1}^L C h_{\ell}^{2s}, 
\]
for a constant $C$ independent of $\{ h_\ell \}$. If we choose $h_0$ sufficiently small such that $\sum_{\ell=1}^\infty h_{\ell}^{2s} < b/C$, this lower bound on $\widehat Z_{h,\{N_{\ell}\}}^\mathrm{ML}$ is positive and independent of $\{h_\ell\}$, $\{N_\ell\}$ and $L$. The claim of the Lemma then follows also for MLMC estimators.
\end{proof}

Using Lemma \ref{lem:rat_linfty} and H\"older's inequality, it then follows from \eqref{eq:rat_mse2} that
\[
e\left(\widehat Q_{h} \Big/ \widehat Z_{h}\right)^2 
\leq 2/Z^2 \, \max\{1,\| \widehat Q_{h} / \widehat Z_{h}\|^{2}_{L^\infty}\} \, \left( \EE\left[ (Q - \widehat Q_{h})^2 \right] + \EE\left[ (Z - \widehat Z_{h})^2 \right] \right)\,.
\]
Thus, the MSE of the ratio $\widehat Q_{h} \big/\widehat Z_{h}$ can be bounded by the sum of the MSEs of $\widehat Q_{h}$ and $\widehat Z_{h}$. 
Using the fact that, {for $Q_h$ the mean of the estimator $\widehat Q_{h}$},
\begin{equation}
\EE\left[ (Q - \widehat Q_{h})^2 \right]  = \left(\EE[ Q - Q_{h}] \right)^2 +  \VV\left[\widehat Q_{h}\right]
\end{equation}
and the results from Sections \ref{sec:fe} and \ref{sec:samp}, this gives the following bounds on the MSEs.

\begin{theorem}\label{thm:mse_unif} Suppose the relevant assumptions of 
{Proposition} \ref{thm:spat_reg}, Lemma \ref{lem:qmc_var_linear} and Lemma \ref{lem:rat_linfty} hold in each case. Then\vspace{-1ex}
\begin{align*}
e\left(\widehat Q_{h,N}^\mathrm{MC}\Big/\widehat Z_{h,N}^\mathrm{MC}\right)^2 &\leq C_\mathrm{MC} \left( N^{-1} + h^{4s}\right), \\
e\left(\widehat Q_{h,N}^\mathrm{QMC} \Big/ \widehat Z_{h,N}^\mathrm{QMC}\right)^2 &\leq C_\mathrm{QMC} \left( N^{-1/\delta} + h^{4s}\right), \\
e\left(\widehat Q_{h,\{N_\ell\}}^\mathrm{ML} \Big/ \widehat Z_{h,\{N_\ell\}}^\mathrm{ML}\right)^2 &\leq C_\mathrm{ML} \bigg( \sum_{\ell=0}^L \frac{h_\ell^{4s}}{N_\ell} + h^{4s}\bigg),
\end{align*}
for some $1/2 < \delta \leq 1$ and for some $0 < s \le 1$, related to the spatial regularity of the data (cf. {Proposition} \ref{thm:spat_reg}), and for constants $C_\mathrm{MC}$,  $C_\mathrm{QMC}$ and $C_\mathrm{ML}$ independent of {$h, N, \{h_\ell\}, \{N_\ell\}$ and $L$}.
\end{theorem}

We note that the convergence rates of the mean square errors in Theorem \ref{thm:mse_unif} are identical to the convergence rates obtained for the individual prior estimators $\widehat Q_h$ and $\widehat Z_h$. 

\subsection{Non-uniformly elliptic case}\label{ssec:mse_gauss}

If the functionals $\phi$ and $\mathcal H$ are uniformly bounded, in the sense that 
\begin{equation}
\label{def:uniformbound}
|\phi(v)| \leq M_1 \ \ \text{and} \ \ |\mathcal H(v)| \leq M_2, \ \ \text{for all} \ \ v \in H^1_0(D), \ \ \text{and} \ \ |\theta(p) - \theta(p_h)| \leq M_3 h,
\end{equation}  
for some constants $M_1, M_2, M_3 > 0$, then the analysis in Lemma \ref{lem:rat_linfty} carries over to the non-uniformly elliptic case, with only minor modifications in the proof.

For more general functionals $\phi$ and $\mathcal H$, the analysis is significantly more difficult and we are only able to analyse Monte Carlo estimators. We restrict to functionals $\phi$ and $\mathcal H$ that satisfy the {polynomial growth conditions} in Assumption A4. Then we can follow an approach similar to \cite{dl09} to obtain the following integrability result on $\widehat Q_{h,N}^\mathrm{MC}/\widehat Z_{h,N}^\mathrm{MC}$.

\begin{lemma}\label{lem:rat_mc_lr} Suppose that $\phi$ and $\mathcal H$ satisfy Assumption A4 and that the same $N$ i.i.d. samples $\{\xi_{\underline J}^{(i)}\}_{i=1}^N$ are used in the estimators $\widehat Q_{h,N}^\mathrm{MC}$ and $\widehat Z_{h,N}^\mathrm{MC}$. Then $\widehat Q_{h,N}^\mathrm{MC} \big/ \widehat Z_{h,N}^\mathrm{MC} \in L^r_\mathbb{P}(\R^J)$ and $\|\widehat Q_{h,N}^\mathrm{MC} \big/ \widehat Z_{h,N}^\mathrm{MC}\|_{L^r_\mathbb{P}(\R^J)}$ can be bounded independent of $h$ and $N$, for all $1 \leq r < \infty$. 
\end{lemma}
\begin{proof}
To simplify the presentation, let us denote $p_h^{i} = p_h(\cdot ; \xi_{\underline J}^{(i)})$.
Then, due to \eqref{eq:fe_h1} and \eqref{polygrowth} we can bound $\theta(p_h^{(i)}) > 0$, $\mathbb P$-almost surely. Hence, it follows from \eqref{eq:mc_est} that
\begin{equation}\label{eq:ratmc}
\left| \frac{\widehat Q_{h,N}^\mathrm{MC}}{\widehat Z_{h,N}^\mathrm{MC}}  \right|= \left|\frac{\frac{1}{N} \sum_{i=1}^N \phi(p_h^{i}) \theta(p_h^{i})}{\frac{1}{N} \sum_{i=1}^N \theta(p_h^{i})}\right| \leq \max_{1 \leq i \leq N} |\phi(p_h^{i})|.
\end{equation}
Using the same argument as in the proof of \cite[Lemma 1]{dl09}, for any convex and non-decreasing function $\rho: \mathbb R_+ \rightarrow \mathbb R_+$, we have
\[
\rho\left(\EE\Big[ \max_{1 \leq i \leq N} \big|\phi(p_h^{i})\big|^{r} \Big] \right) \leq \EE \left[ \rho\Big(\max_{1 \leq i \leq N} \big|\phi(p_h^{i})\big|^{r}\Big)\right] 
\leq \sum_{i=1}^{N} \EE\left[\rho\Big(\big|\phi(p_h^{i})\big|^{r}\Big)\right]. 
\]
Choosing $\rho(x) = |x|^{\tilde r/r}$, for some $r \le \tilde r < \infty$, and using \eqref{polygrowth} we have 
\[
\Big \| \max_{1 \leq i \leq N} |\phi(p_h^{i})| \Big\|_{L^{r}_\mathbb{P}(\R^J)} \leq \|\phi(p_h)\|_{L^{\tilde r}_\mathbb{P}(\R^J)} N^{1/\tilde r} \leq M_1 \left( 1 + \|f\|^{n_2}_{H^{-1}(D)} \|k_\mathrm{min}^{-1}\|^{n_2}_{L^{n_2 \tilde r}_\mathbb{P}(\R^J)}\right) N^{1/\tilde r}\,,
\]
The term in the bracket is finite due to Lemma \ref{lem:gauss_exp}  and the claim of the Lemma now follows if we choose $\tilde r \ge \text{ln} \, N$.
\end{proof}

\begin{theorem}\label{thm:mse_nonunif}
Suppose the assumptions of {Proposition} \ref{thm:spat_reg} and Lemma \ref{lem:rat_mc_lr} hold. Then
\begin{equation*}
e\left(\widehat Q_{h,N}^\mathrm{MC}\Big/\widehat Z_{h,N}^\mathrm{MC}\right)^2
\leq C_\mathrm{MC} \left( N^{-1} + h^{4s}\right)\,,
\end{equation*}
for some $0 < s \le 1$ related to the spatial regularity of the data 
(cf.~{Proposition} \ref{thm:spat_reg}) and for a constant $C_\mathrm{MC}>0$ independent of $h$ and $N$.
\end{theorem}
\begin{proof}
Since $\widehat Q_{h,N}^\mathrm{MC} \big/ \widehat Z_{h,N}^\mathrm{MC}$ is not in $L^\infty_\mathbb{P}(\R^J)$ in this case, we apply the Cauchy-Schwarz inequality to the second term on the right hand side of \eqref{eq:rat_mse2} to obtain
\begin{equation}
\label{ineq:thm_nonunif}
e\left(\widehat Q_{h,N}^\mathrm{MC} \Big/ \widehat Z_{h,N}^\mathrm{MC} \right)^2  
\le \frac{2}{Z^2} \bigg(\EE\left[ (\widehat Q_{h,N}^\mathrm{MC} - Q )^2\right]  + \|\widehat Q_{h,N}^\mathrm{MC} \big/ \widehat Z_{h,N}^\mathrm{MC}\|^2_{L^4_\mathbb{P}(\R^J)} \|\widehat Z_{h,N}^\mathrm{MC} - Z \|^2_{L^4_\mathbb{P}(\R^J)}\bigg). 
\end{equation}
To bound $\| \widehat Z_{h,N}^\mathrm{MC} - Z \|_{L^4_\mathbb{P}(\R^J)}$ we apply the triangle inequality and consider separately $\| Z_h - Z \|_{L^4_\mathbb{P}(\R^J)}$ and $\| \widehat Z_{h,N}^\mathrm{MC} - Z_h \|_{L^4_\mathbb{P}(\R^J)}$. The former is bounded by $Ch^{2s}$ due to Theorem \ref{thm:fe_post}. 

To bound the latter, let $X_i := \theta(p_h(\cdot;\xi_{\underline J}^{(i)})$, $i=1,\ldots,N$. Since the range of $\theta$ is $[0,1]$, this is a sequence of i.i.d. random variables with finite mean $m := Z_h$, finite variance $\sigma_X^2 := \EE[(\theta(p_h)-Z_h)^2]$ and finite central fourth moment $\tau_X^4 := \EE[(\theta(p_h)-Z_h)^4]$. A direct calculation gives
\[
\| \widehat Z_{h,N}^\mathrm{MC} - Z_h \|_{L^4_\mathbb{P}(\R^J)}^4 = \EE\left[ \left( \widehat X_N - m \right)^4 \right] \, = \, \frac{3 (N-1)\sigma_X^4/2 + \tau_X^4}{N^3} \, \leq \, \frac{3 \sigma_X^4/2 + \tau_X^4/N}{N^2}.
\]
Thus, the result follows from \eqref{ineq:thm_nonunif} together with Theorem \ref{thm:fe_post} and Lemma \ref{lem:rat_mc_lr}.
\end{proof}

The proofs of Lemma \ref{lem:rat_mc_lr} and Theorem \ref{thm:mse_nonunif} can potentially also be extended to the case of Quasi-Monte Carlo estimators. However, it seems impossible to satisfy Assumption A4 in the case of non-uniformly elliptic coefficients $k$ and so we did not pursue this any further.

The analysis in the case of multilevel Monte Carlo estimators is complicated by the fact that the multilevel estimator $\widehat Z_{h,\{N_\ell\}}^\mathrm{ML}$ can take on negative values. This precludes the approach in the proof of Lemma \ref{lem:rat_mc_lr}. The existence of moments of a ratio of random variables where the denominator is not strictly positive has been the subject of research since the 1930s \cite{ceary30,fieller32,hinkley69,oa94,kl01,gp12}, and is a problem not yet fully solved. A possible approach to show existence of moments of the ratio $\widehat Q_{h,\{N_\ell\}}^\mathrm{ML}\big/\widehat Z_{h,\{N_\ell\}}^\mathrm{ML}$ could be to use the Central Limit Theorem in \cite{chnst14}, which shows that the individual MLMC estimators are asymptotically normally distributed as the number of levels and the number of samples per level tend to infinity. Hinkley \cite{hinkley69} then gives an explicit expression of the cumulative distribution function of the ratio of two correlated normal random variables, together with its limiting  normal distribution, as the denominator tends to a normal random variable with non-zero mean and zero variance.

\subsection{Computational $\varepsilon$-cost}
Based on the bounds on the mean square errors given in Theorems \ref{thm:mse_unif} and \ref{thm:mse_nonunif}, we now analyse the computational complexity of the various estimators of our quantity of interest $Q/Z$. We are interested in bounding the $\varepsilon$-cost, i.e., the cost required to achieve a MSE of order $\varepsilon^2$. Since the convergence rates of the mean square error are the same as for the individual estimators $\widehat Q_h$ and $\widehat Z_h$, bounds on the computational $\varepsilon$-cost can be proved as in \cite{tsgu13, gknsss15}. 

We denote by $C_\ell$ the cost of obtaining one sample of $\theta(p_{h_\ell})$ and/or $\psi(p_{h_\ell})$. { This cost will typically also depend on the truncation parameter $J$, but we will not make this dependence explicit here. }
We furthermore denote by $\mathcal{C}_\mathrm{MC}$, $\mathcal{C}_\mathrm{ML}$ and $\mathcal{C}_\mathrm{QMC}$ the computational cost of the ratio estimator $\widehat Q_h \big/ \widehat Z_h$ based on MC, MLMC and QMC estimators, respectively.

\begin{theorem}\label{thm:comp} 
Let the conclusions of Theorem \ref{thm:mse_unif} or Theorem \ref{thm:mse_nonunif} hold and suppose 
\[
\mathcal{C}_\ell \leq C_\gamma h_\ell^{-\gamma}, \quad \text{for some} \ \  \gamma > 0.
\]
Then for any $\varepsilon < e^{-1}$, there exist a constant $C^{\mathrm{ML}}>0$, a value $L \in \mathbb{N}$ and a sequence $\{N_\ell\}_{\ell=0}^{L}$, such that\vspace{-2ex}
\begin{equation*} 
\qquad e\left(\widehat Q_{h,\{N_\ell\}}^\mathrm{ML}/\widehat Z_{h,\{N_\ell\}}^\mathrm{ML}\right)^2 \leq \varepsilon^2\quad \text{and} \quad \mathcal{C}_\mathrm{ML} \leq \left\{ 
\begin{array}{ll}
         C^{\mathrm{ML}}\varepsilon^{-2}, & \mbox{if \ $s < \gamma/4$},\\
		C^{\mathrm{ML}}\varepsilon^{-2} (\log \varepsilon)^2, & \mbox{if \ $s = \gamma/4$},\\
        C^{\mathrm{ML}}\varepsilon^{-\gamma/2s}, & \mbox{if \ $s > \gamma/4$},\end{array}
\right.
\end{equation*}
where $0 < s \le 1$ is related to the spatial regularity of the data (cf. {Proposition} \ref{thm:spat_reg}). Furthermore, there exist positive constants $C^{\mathrm{MC}}$ and $C^{\mathrm{QMC}}$ and values of $h$ and $N$, such that 
\begin{align*}
e\left(\widehat Q_{h,N}^\mathrm{MC}/\widehat Z_{h,N}^\mathrm{MC}\right)^2 \leq \varepsilon^2 \quad &\text{and} \quad \mathcal{C}_\mathrm{MC} \leq C^{\mathrm{MC}} \; \varepsilon^{-2-\gamma/2s}, \\
\qquad e\left(\widehat Q_{h,N}^\mathrm{QMC}/\widehat Z_{h,N}^\mathrm{QMC}\right)^2 \leq \varepsilon^2 \quad &\text{and} \quad
 \mathcal{C}_\mathrm{QMC} \leq C^{\mathrm{QMC}} \; \varepsilon^{-2\delta-\gamma/2s}, \quad \text{for some} \ \ 1/2 < \delta \le 1.
 \end{align*}
\end{theorem}
{Theorem \ref{thm:comp} shows that MLMC and QMC can outperform standard MC in terms of the growth rate of the $\varepsilon$-cost. However, both QMC and MLMC also require stronger assumptions than MC (cf section \ref{sec:samp}).}

\section{Numerical examples}\label{sec:num}
We now study the performance of the ratio estimators on a typical model problem. As the forward model, we take the elliptic equation
\begin{equation}\label{eq:mod_num}
{-}\nabla \cdot (k(x;\xi_{\underline J})) \nabla p(x;\xi_{\underline J})) = 0, \qquad \text{in } D = (0,1)^2,
\end{equation}
subject to the deterministic, mixed boundary conditions $p|_{x_1=0}=1, p|_{x_1=1}=0$ and zero Neumann conditions on the remainder of the boundary. The prior distribution on the coefficients $\xi_{\underline J}$ is Gaussian, as in Section \ref{ssec:invprob_para_gauss}, and $k$ is a (truncated) log-normal random field. We choose the exponential covariance function \eqref{eq:cov_exp} with $r=1$, correlation length $\lambda=0.3$ and variance $\sigma^2=1$. The mean $m_0$ is chosen to be 0. In this case, the assumptions of {Proposition} \ref{thm:spat_reg} hold for any $t < 1/2$.

For the spatial discretisation, we use standard, continuous, piecewise linear finite elements on a uniform triangular mesh. The stiffness matrix is assembled using the trapezoidal rule for quadrature. The mesh hierarchy for the MLMC estimator is generated by uniform refinement of a uniform grid with coarsest mesh width $h_0=1/8$, and $h_{\ell-1}/h_{\ell} = 2$, for all $\ell = 1, \dots, L$. 

The quantity of interest $\phi$ is the outflow over the boundary at $x_1 = 1$. To obtain optimal convergence rates of the finite element error, we compute $\phi(p_h)$ as
\[
\phi(p_h) = - \int_D k(x;\xi_{\underline J}) \nabla w_h(x) \cdot \nabla p_h(x;\xi_{\underline J})) \mathrm{d} x,
\]
for a suitably chosen weight function $w_h$ with $w_h|_{x_1=0}=0, w_h|_{x_1=1}=1$ \cite{tsgu13}.  In particular, we choose $w_h \in V_h$ to be one at the nodes of the finite element mesh on the boundary $x_1 = 1$ and zero at all other nodes.

The data $y$ is generated from the solution of equation \eqref{eq:mod_num} with a random sample $\xi_{\underline J}$ from the prior distribution, on a fine reference mesh with $h^* = 1/256$. The observation functional $\mathcal H$ is taken as a local average pressure, representing a regularised point evaluation. To obtain $m$-dimensional data $y$, we take the uniform finite element mesh on $[0,1]^2$ with grid size $1/(\sqrt{m}+1)$, and evaluate the local average pressure at the $m$ interior nodes in this mesh. The average is taken over the six elements of the finite element mesh with $h^* = 1/256$ adjacent to that node. We furthermore add observational noise to the data $y$, which is a realisation of {an} $m$-dimensional normal random variable with mean zero and covariance $\sigma_\eta^2 I$.

To generate samples of $k$, we use a truncated Karhunen-Lo\`eve expansion \cite{ghanem_spanos}, i.e. we truncate the infinite expansion \eqref{eq:def_upara} at a finite order $ J_{KL} = 1400$. An alternative that allows to sample from the infinite expansion \eqref{eq:def_upara} would be the circulant embedding method \cite{dn97,gknss11,gknss15}.

For the QMC estimators, we choose a lattice rule with product weight parameters $\gamma_j=1/j^2$ and one random shift. The generating vector for the rule used is available from Frances Kuo's website (\texttt{http://web.maths.unsw.edu.au/$\sim$fkuo/}) as "lattice-39102-1024-1048576.3600''. {We point out here that this generating vector is a standard, off the shelf generating vector, rather than a generating vector specifically constructed for the weights implicitly defined in Assumption A4 and Lemma \ref{lem:qmc_var_linear}. 
In practice we found this generating vector to work well, even though the convergence rates in Lemma \ref{lem:qmc_var_linear} were not proven for this particular choice.}

\subsection{Mean Square Error}\label{ssec:num_mse}
We start by investigating the discretisation error, the sampling error and the mean square error of the ratio estimators, for a fixed number of observations $m = 9$ and a fixed level of observational noise $\sigma_\eta^2 = 0.09$. 

\begin{figure}[t]
\centering
\hspace*{-0.75cm}\includegraphics[width=0.5\textwidth]{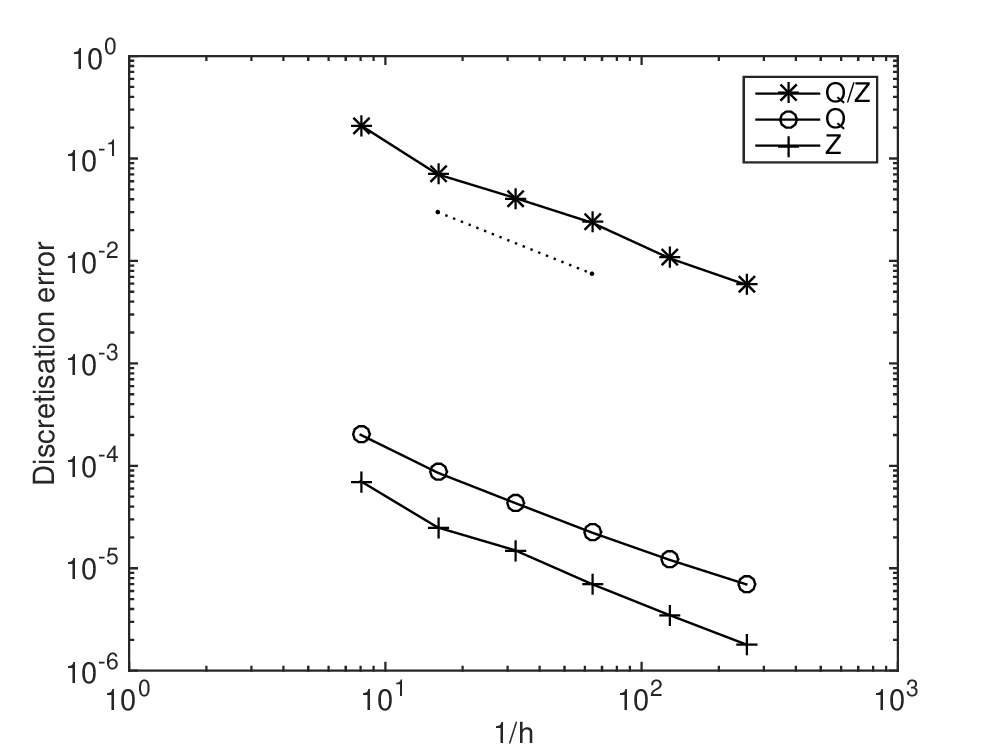}\ \includegraphics[width=0.5\textwidth]{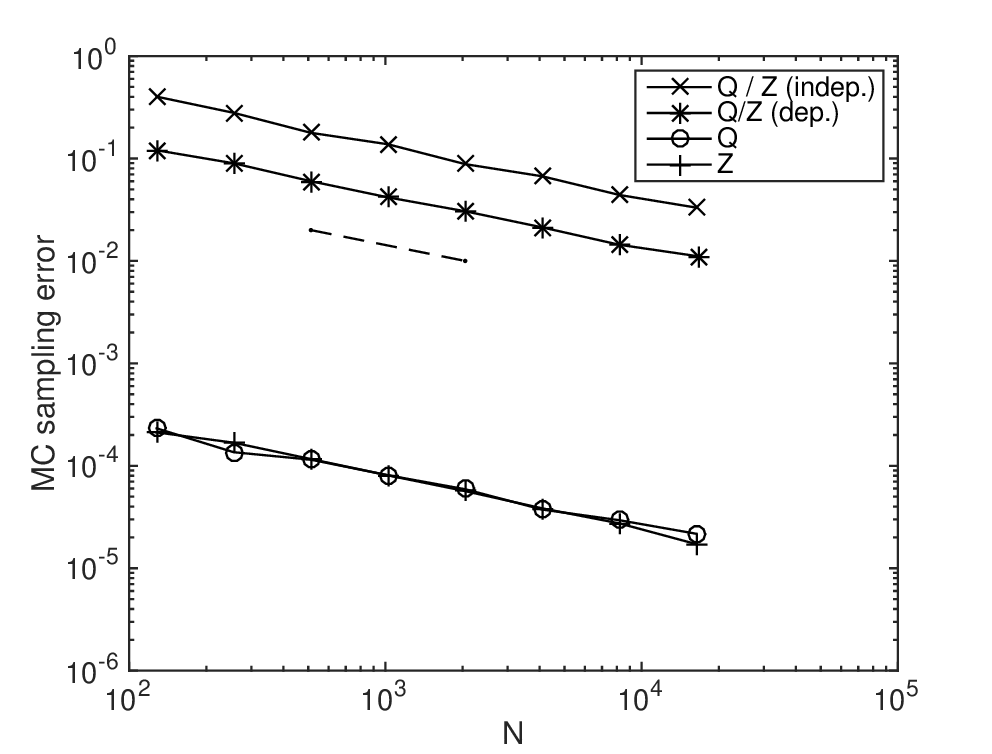}\hspace*{-0.75cm} \\
\hspace*{-0.75cm}\includegraphics[width=0.5\textwidth]{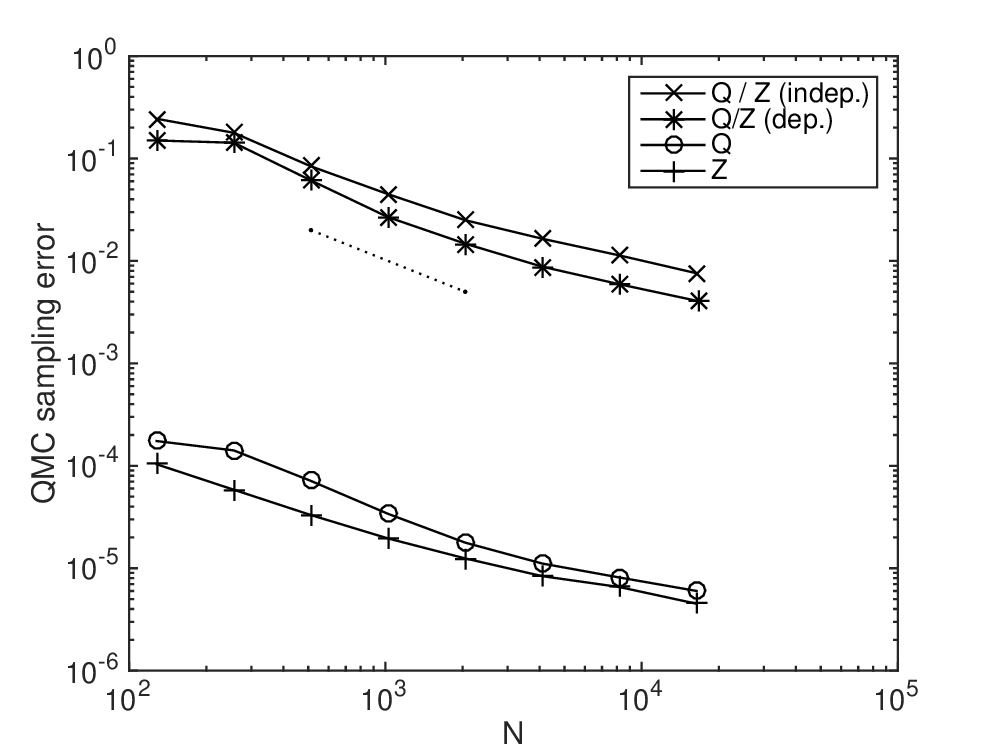}\ \includegraphics[width=0.5\textwidth]{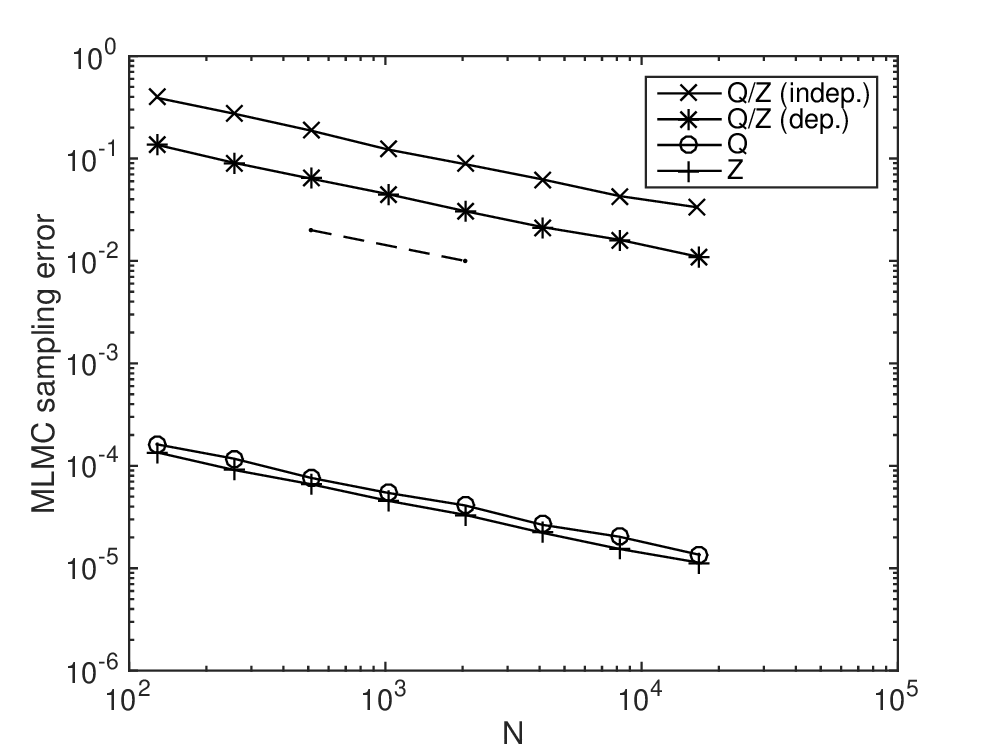}\hspace*{-0.75cm}
\caption{Convergence w.r.t. $h$ of the discretisation errors $|Q_h/Z_h - Q_{2h}/Z_{2h}|$, $|Q_h - Q_{2h}|$ and $|Z_h - Z_{2h}|$ (top left), as well as convergence w.r.t. $N$ of the sampling errors $\EE[(\widehat Q_h / \widehat Z_h - Q_h / Z_h)^2]^{1/2}$, 
$\EE[(\widehat Q_h - Q_h )^2]^{1/2}$ and $\EE[(\widehat Z_h - Z_h )^2]^{1/2}$ for MC (top right), QMC (bottom left) and MLMC (bottom right), respectively. The dotted and dashed reference slopes are $-1$ and $-1/2$, respectively.}
\label{fig:disc_samp}
\end{figure} 
 Figure \ref{fig:disc_samp} shows the discretisation error and the sampling errors of the different estimators. The top left plot shows the discretisation error $|Q_h/Z_h - Q_{2h}/Z_{2h}|$, as well as the individual discretisation errors $|Q_h - Q_{2h}|$ and $|Z_h - Z_{2h}|$. We see these errors decay linearly in $h$, as predicted by Theorem~\ref{thm:fe_post}. 

The other three plots show the sampling error $\EE[(\widehat Q_h / \widehat Z_h - Q_h / Z_h)^2]^{1/2}$, as well as the individual sampling errors $\EE[(\widehat Q_h - Q_h )^2]^{1/2}$ and $\EE[(\widehat Z_h - Z_h )^2]^{1/2}$, for MC (top right), QMC (bottom left) and MLMC (bottom right). The mesh size $h$ is fixed at $h=1/16$, and the ``exact'' expected values $Q_h$ and $Z_h$ are estimated with MLMC with a very large number of samples. For MC and QMC, $N$ on the horizontal axis represents the number of samples. For MLMC, $N$ represents the equivalent number of solves on the finest grid $h=1/16$ that would lead to the same cost as the MLMC estimator. This means that for a given $N$, the cost of all three estimators is the same. The number of samples $N_\ell$ in the MLMC estimator was chosen proportional to $h_\ell^{-(4s+\gamma)/2} \approx h_\ell^{-2}$, as suggested by the optimisation in \cite{giles08,cgst11}, assuming $s\approx1/2$ and $\gamma \approx 2$. We show results for ratio estimators with the same random samples used in $\widehat Q_h$ and $\widehat Z_h$, referred to as dependent estimators, as well as ratio estimators with different random samples used in $\widehat Q_h$ and $\widehat Z_h$, referred to as independent estimators. 
For MC and MLMC, we observe a convergence rate of $N^{-1/2}$. For QMC, we observe a convergence rate which is significantly faster than order $N^{-1/2}$ and almost order $N^{-1}$. 

Figure \ref{fig:mse} compares the computational costs of the different estimators to achieve an RMSE of $\varepsilon$. The computational cost of the estimators was computed as $N h^{-2}$ for the MC and QMC estimators, and as $N_0 h_0^{-2} + \sum_{\ell=1}^L N_\ell (h_\ell^{-2} + h_{\ell-1}^{-2})$ for the MLMC estimator. The bias $|Q_h/Z_h - Q/Z|$ was estimated from the values of $|Q_h/Z_h - Q_{2h}/Z_{2h}|$ shown in Figure \ref{fig:disc_samp}.
As predicted by Theorem \ref{thm:mse_unif} with $\gamma\approx2$ and $s \approx 1/2$, the cost of the MC estimator grows with about $\varepsilon^{-4}$, the cost of the QMC estimator grows with about $\varepsilon^{-3}$, and the cost of the MLMC estimator grows with about $\varepsilon^{-2}$. 

\begin{figure}[t]
\centering
\hspace*{-0.75cm}\includegraphics[width=0.5\textwidth]{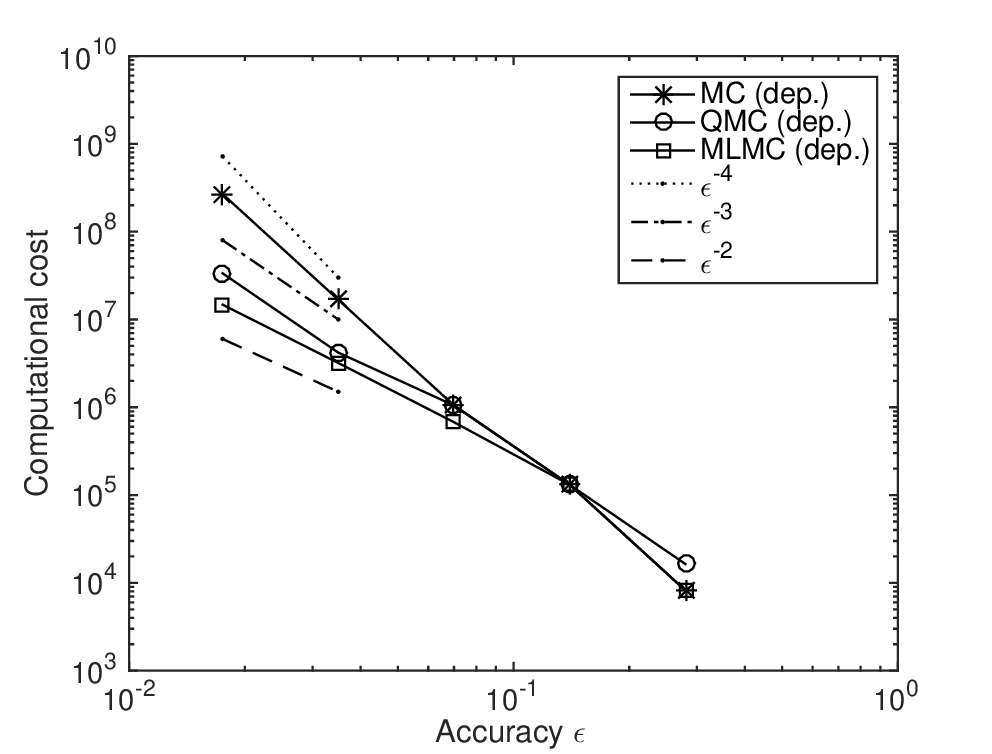}\ \includegraphics[width=0.5\textwidth]{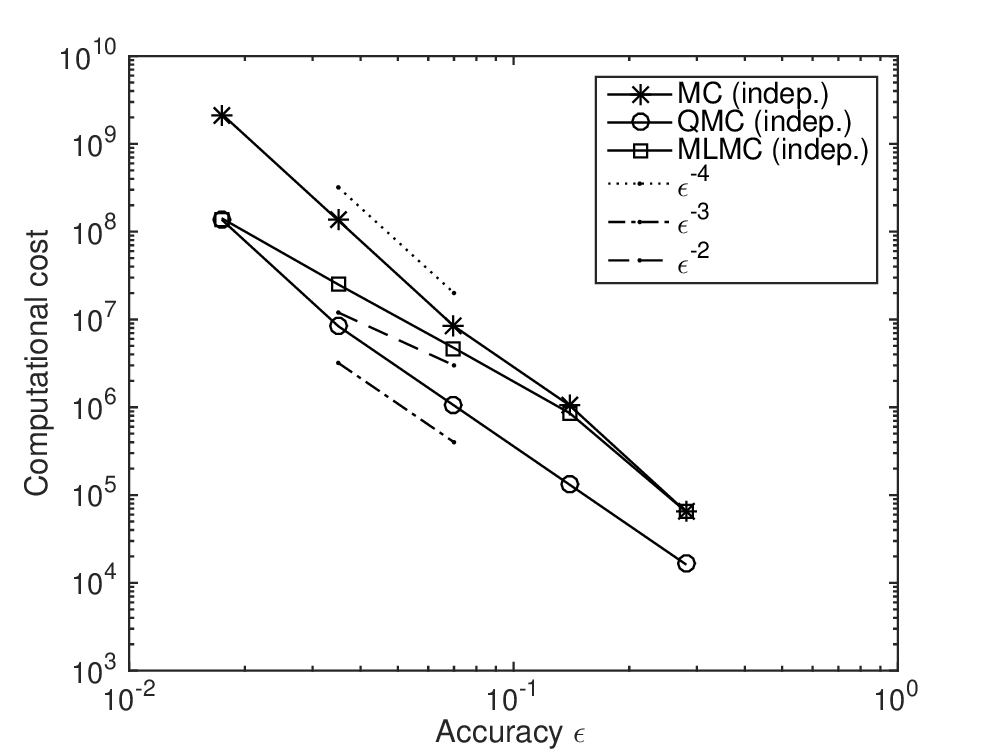}\hspace*{-0.75cm} 
\caption{Computational cost of ratio estimators $\widehat Q_h/ \widehat Z_h$ to achieve a RMSE $e(\widehat Q_h/ \widehat Z_h)$ of $\varepsilon$, using the same random samples in $\widehat Q_h$ and in $\widehat Z_h$ (left) and using different random samples in $\widehat Q_h$ and in $\widehat Z_h$ (right), respectively.}
\label{fig:mse}
\end{figure}

\subsection{Dependency on $m$ and $\sigma_\eta^2$}
Finally, we look at the dependency of the sampling error on the number of observations $m$ and on the noise level $\sigma_\eta^2$. For large values of $m$ and small values of $\sigma_\eta^2$, we expect the posterior distribution $\mu^y$ to concentrate on a small region of the parameter space. The ratio estimators sample from the prior distribution and do not make use of this fact. We expect the sampling errors to grow with increasing $m$ and decreasing $\sigma_\eta^2$. To ameliorate this problem, one can under certain assumptions rescale the parameter space before applying the ratio estimators, see e.g.  \cite{ss15} for details. 
\begin{figure}[t]
\centering
\hspace*{-0.75cm}\includegraphics[width=0.5\textwidth]{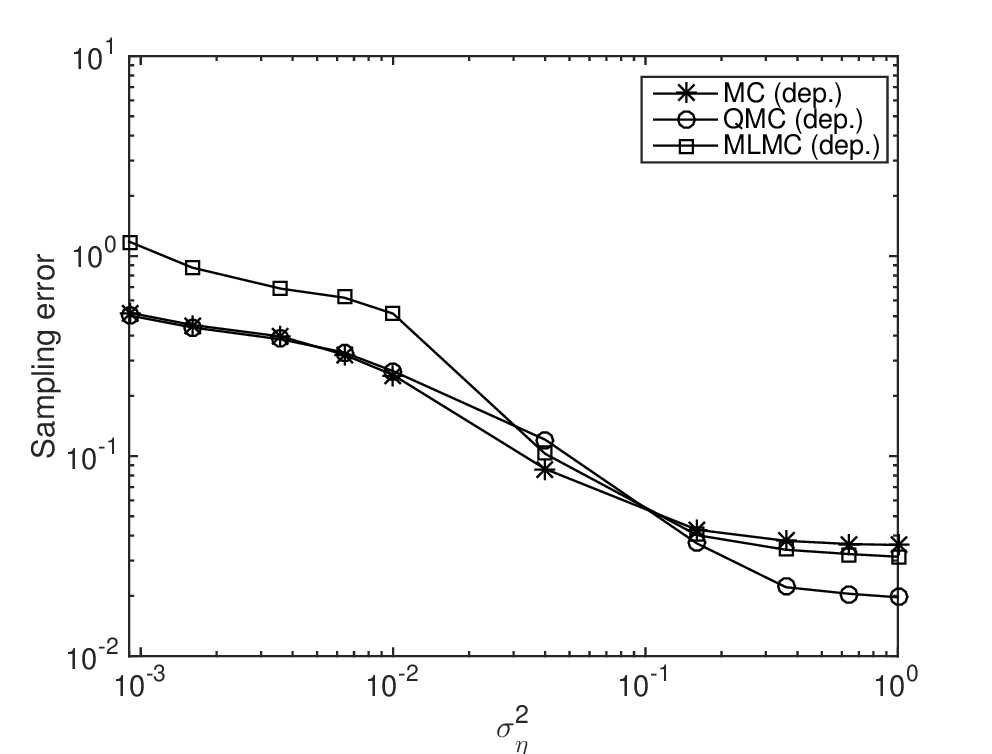}\ \includegraphics[width=0.5\textwidth]{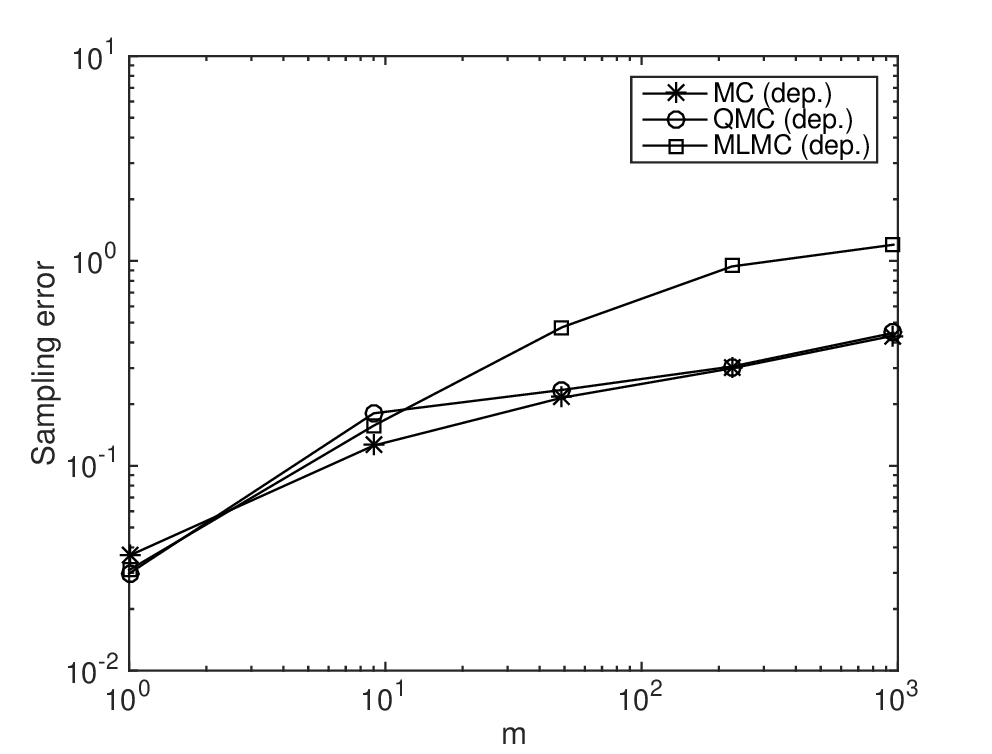}\hspace*{-0.75cm} 
\caption{Sampling errors $\EE[(\widehat Q_h / \widehat Z_h - Q_h / Z_h)^2]^{1/2}$ as a function of noise level $\sigma_\eta^2$ (left) and as a function of number of observations $m$ (right), respectively.}
\label{fig:obs_noise}
\end{figure} 

Figure \ref{fig:obs_noise} shows the sampling error of ratio estimators based on using the same samples in $\widehat Q_h$ and $\widehat Z_h$. We observe a mild growth of the sampling errors both with increasing $m$ and decreasing $\sigma_\eta^2$, but the growth is not dramatic and all estimators appear to be robust over a large range of practically interesting values. The fact that the sampling error for MLMC based estimators grows more quickly than for MC and QMC based estimators, is at least partly caused by the fact that $\VV[\psi_{h_0}]$ and $\VV[\psi_{h_1} - \psi_{h_0}]$ become of the same size for small $\sigma_\eta^2$ or large $m$, making the choice $N_\ell = C h_\ell^{-2}$ less and less optimal. Experiments with estimators based on using independent samples in $\widehat Q_h$ and $\widehat Z_h$ also showed growth of sampling errors for small $\sigma_\eta^2$ and large $m$, in fact much faster than in the case of dependent estimators.

\section{Conclusions and further work}\label{sec:conc}
In Bayesian inverse problems, the goal is often to compute the expected value of a quantity of interest under the posterior distribution. For sampling based approaches, one has to overcome the difficulty that the posterior distribution is typically intractable, in the sense that direct sampling from it is unavailable since the normalisation constant is unknown. We considered here an approach based on Bayes' theorem that computes an estimate of the normalisation constant and estimates the posterior expectation as the ratio of two prior expectations. To compute the prior expectations, we considered the sampling based approaches of Monte Carlo, quasi-Monte Carlo and multilevel Monte Carlo estimators. For a model elliptic inverse problem, we provided a full convergence and complexity analysis of the resulting ratio estimators. Our theory shows that asymptotically the complexity of computing the posterior expectation with this approach is the same as computing prior expectations, and this result is also confirmed numerically for a typical model problem in uncertainty quantification.

It would be interesting to compare the performance of the ratio estimators considered in this work to Markov chain Monte Carlo (MCMC) and multilevel Markov chain Monte Carlo (MLMCMC) methods \cite{hss13,kst13}. Especially in the case of small noise level $\sigma_\eta^2$ or large number of observations $m$, MCMC based approaches might explore the posterior distribution more efficiently. In terms of the $\varepsilon$-cost of estimators, the analysis and simulations in \cite{kst13} show that the computational cost of a standard MCMC estimator grows at the same rate as a ratio estimator based on MC, and the cost of an MLMCMC estimator will grow at the same rate as a ratio estimator based on MLMC. The constants appearing in these estimates, for MCMC based approaches, depend on quantities like the acceptance rate, autocorrelation and possibly the dimension of the parameter space. For the high dimensional problems considered in this work, these constants might be very large.

\begin{appendix}

\section{Proof of Assumption A4 for $\theta$ for linear and scalar $\mathcal{H}$}
\label{sec:appendix}

Let $m=1$, let $\mathcal{H}$ be a linear functional on $V = H^1_0(D)$ and let $k^*_{\min} := \min_{x \in \overline{D}} k^*(x) > 0$ and thus $k_{\min}(\xi_{\underline J}) \ge k^*_{\min} > 0$, for all $\xi_{\underline J} \in \mathbb{R}^J$. 
Let us assume without loss of generality that $y = 0$ and $\sigma_\eta^2 =1$, and for simplicity let $\xi = \xi_{\underline J} \in \mathbb{R}^J$. Then, $\theta(p_h(\cdot;\xi) = g(h(\xi))$ with $g(\zeta) := \exp(-\zeta^2/2) \le 1$ and $h(\xi) := \mathcal{H}(p_h(\cdot;\xi))$. To simplify the presentation, we write $\displaystyle g_n := \frac{\text{d}^ng}{\text{d}\zeta^n}(h(\xi))$ and $\displaystyle h_{\bmu} := \frac{\partial^{|\bmu|} h}{\partial \xi^{\bmu}}(\xi)$ where $\bmu$ is a multi-index in $\{0,1\}^J$. 

Let $C$ be a generic constant independent of $J$, $\bnu$ and $\xi$.
First note that due to \eqref{eq:fe_h1}  and the linearity of $\mathcal{H}$
\[
|h(\xi)| \le \|\mathcal{H}\|_{H^{-1}(D)} |p_h|_{H^1_0(D)} \le \frac{\|\mathcal{H}\|_{H^{-1}(D)} \|f\|_{H^{-1}(D)}}{k^*_{\min}} =: \kappa_*.
\]
Then, we have $g_n(\zeta) = (-1)^nH_n(\zeta) g(\zeta)$ where $H_n$ is the $n$th Hermite polynomial, and so
\begin{equation}
\label{eq:g_n_bound}
|g_n(h(\xi))| \le C \max \{1,|h(\xi)|^n\} \, g(h(\xi)) \le C \max \{1,\kappa_*^n\}\,.
\end{equation}
Moreover, it was shown in \cite[Theorem 16]{gknsss15} that, for linear $\mathcal{H}$,
\begin{equation}
\label{eq:h_bmu_bound}
|h_{\bmu} (\xi)| \le \kappa_* \frac{|\bmu|!}{(\ln 2)^{|\bmu|}} \prod_{j=1}^J b_j^{\mu_j}\,.
\end{equation}

Now, Faa di Bruno's formula for the special case where $\bnu \in \{0,1\}^J$ and where $h(\xi)$ is scalar (cf. \cite[Corollary 2.10]{cs96}) states that
\[
\theta_{\bnu} = \sum_{r=1}^{|\bnu|} \; g_r \sum_{P(r,\bnu)} \; \prod_{i=1}^r h_{\bmu^{(i)}}
\]
where
\[
P(r,\bnu) := \left\{\bmu^{(1)},\ldots,\bmu^{(r)} : \boldsymbol{0} \prec \bmu^{(1)} \prec \ldots \prec \bmu^{(r)} \text{ and } \sum_{i=1}^r \bmu^{(i)} = \bnu\right\}
\] 
and $\prec$ indicates some unique linear ordering of multi-indices (see \cite[p.~505]{cs96} for an example). And so, using \eqref{eq:g_n_bound} and \eqref{eq:h_bmu_bound}, we get
\begin{align*}
|\theta_{\bnu}| & \le C \left| \sum_{r=1}^{|\bnu|} \; \max\{1,\kappa_*^r\} \sum_{P(r,\bnu)}  \prod_{i=1}^r \left(\kappa_* \frac{|\bmu^{(i)}|!}{(\ln 2)^{|\bmu^{(i)}|}} \prod_{j=1}^J b_j^{\mu^{(i)}_j} \right)\right| \\
& \le C \left(\frac{\max\{\kappa_*,\kappa_*^2\}}{\ln 2}\right)^{|\bnu|} \underbrace{\left( \sum_{r=1}^{|\bnu|} \sum_{P(r,\bnu)} \prod_{i=1}^r |\bmu^{(i)}|! \right)}_{=:\rho_{\bnu}} \; \prod_{j=1}^J b_j^{\nu_j}
\end{align*}
since all elements of $P(r,\bnu)$ satisfy $\sum_{i=1}^r \bmu^{(i)} = \bnu$ and $\sum_{i=1}^r |\bmu^{(i)}| = |\bnu|$. 
It remains to bound $\rho_{\bnu}$. We give a simple but fairly crude bound. 

First note that, for each element $(\bmu^{(1)},\ldots,\bmu^{(r)}) \in P(r,\bnu)$, the moduli $|\bmu^{(i)}|$, $i=1,\ldots,r$, form a partition of $n := |\bnu|$. Hence, instead of
partitioning the summands in $\rho_{\bnu}$ into the subsets $P(r,\bnu)$, we can also sum over all possible partitions $k_1,\ldots,k_r$ of $n$ with $1 \le r \le n$. The partition function $p(n)$ is the number of possible partitions of $n$. It is bounded by $\exp(\pi \sqrt{2n/3})$ \cite{azev09}. For each partition $k_1,\ldots,k_r$ of $n$,  the number of possible elements in $P(r,\bnu)$ that satisfy $|\bmu_i| = k_i$ can be bounded by
\[
\left( {n \atop k_1}\right) \left( {n-k_1 \atop k_2}\right) \left( {n-\sum_{i=1}^2 k_i \atop  k_3}\right) \ldots \left( {n-\sum_{i=1}^{r-1} k_i \atop  k_r}\right) = \frac{n!}{k_1 !k_2 ! \ldots k_r !}
\]
Elements of $P(r,\bnu)$ where $|\bmu_i| = |\bmu_j|$, for some $i \not= j$, are counted twice in this bound. Since $\prod_{i=1}^r |\bmu^{(i)}|! = k_1 !k_2 ! \ldots k_r !$, we finally get the bound
\[
\rho_{\bnu} \le p\big(|\bnu|\big) \, |\bnu|! \le \exp\left(\pi \sqrt{\frac{2|\bnu|}{3}}\right) \, |\bnu|! \;.
\]
Hence, there exists a constant $c_p > 1$ such that Assumption A4 holds with $c_1 := c_p \max\{\kappa_*,\kappa_*^2\}/\ln 2$.

\end{appendix}

\small
\bibliographystyle{siam}
\bibliography{bibMLMC}

\begin{thebibliography}{10}

\bibitem{apss15}
{\sc S.~Agapiou, O.~Papaspiliopoulos, D.~Sanz-Alonso, and A.~M. Stuart}, {\em
  Importance sampling: computational complexity and intrinsic dimension}.
\newblock Available as arXiv preprint arXiv:1511.06196.

\bibitem{bsz11}
{\sc A.~Barth, Ch. Schwab, and N.~Zollinger}, {\em Multilevel {M}onte {C}arlo
  finite element method for elliptic {PDE}'s with stochastic coefficients},
  Numer. ~Math., 119 (2011), pp.~123--161.

\bibitem{ceary30}
{\sc R.C. Ceary}, {\em The frequency distribution of the quotient of two normal
  variables}, J. Roy. Statist. Soc, 93 (1930), pp.~442--446.

\bibitem{charrier12}
{\sc J.~Charrier}, {\em Strong and weak error estimates for the solutions of
  elliptic partial differential equations with random coefficients}, SIAM J.
  Numer. Anal, 50 (2012), pp.~216--246.

\bibitem{cst13}
{\sc J.~Charrier, R.~Scheichl, and A.~L. Teckentrup}, {\em Finite element error
  analysis of elliptic {PDE}s with random coefficients and its application to
  multilevel {M}onte {C}arlo methods}, SIAM J.~Numer.~Anal., 51 (2013),
  pp.~322--352.

\bibitem{cgst11}
{\sc K.~A. Cliffe, M.~B. Giles, R.~Scheichl, and A.~L. Teckentrup}, {\em
  Multilevel {M}onte {C}arlo methods and applications to elliptic {PDE}s with
  random coefficients}, Comput.~Vis.~Sci., 14 (2011), pp.~3--15.

\bibitem{cds10}
{\sc A.~Cohen, R.~DeVore, and Ch. Schwab}, {\em Convergence rates of best
  {N}-term {Galerkin} approximations for a class of elliptic {SPDEs}},
  Found.~Comput.~Math., 10 (2010), pp.~615--646.

\bibitem{chnst14}
{\sc N.~Collier, A.-L. Haji-Ali, F.~Nobile, E.~von Schwerin, and R.~Tempone},
  {\em A continuation multilevel {M}onte {C}arlo algorithm}, BIT Numerical
  Mathematics, 55 (2014), pp.~399--432.

\bibitem{cmps14}
{\sc P.~R. Conrad, Y.~M. Marzouk, N.~S. Pillai, and A.~Smith}, {\em
  Accelerating asymptotically exact {MCMC} for computationally intensive models
  via local approximations}, J. Amer. Statist. Assoc.,  (2015).

\bibitem{cs96}
{\sc G.~M. Constantine and T.~H. Savits}, {\em A multivariate {Fa\`a} di
  {Bruno} formula with applications}, Trans.~Amer.~Math.~Soc., 248 (1996),
  pp.~503--520.

\bibitem{crsw13}
{\sc S.~L. Cotter, G.~O.Roberts, A.~M. Stuart, and D.~White}, {\em {MCMC}
  methods for functions: modifying old algorithms to make them faster}, Stat.
  Sci., 28 (2013), pp.~424--446.

\bibitem{ds14}
{\sc M.~Dashti and A.M.Stuart}, {\em The {B}ayesian approach to inverse
  problems}, in Handbook of Uncertainty Quantification, R.~Ghanem, D.~Higdon,
  and H.~Owhadi, eds., Springer, 2015.

\bibitem{azev09}
{\sc W.~de~Azevedo~Pribitkin}, {\em Simple upper bounds for partition
  functions}, Ramanujan J., 18 (2009), pp.~113--119.

\bibitem{dggs16}
{\sc J.~Dick, R.N. Gantner, Q.T. Le~Gia, and Ch. Schwab}, {\em Higher order
  quasi-{M}onte {C}arlo integration for {B}ayesian estimation}, arXiv preprint
  arXiv:1602.07363,  (2016).

\bibitem{dggs16_2}
{\sc J.~Dick, R.~N. Gantner, Q.~T.~Le Gia, and Ch. Schwab}, {\em Multilevel
  higher order {q}uasi-{M}onte {C}arlo {B}ayesian estimation}, Tech. Report
  2016-34, Seminar for Applied Mathematics, ETH Z{\"u}rich, Switzerland, 2016.

\bibitem{dks13}
{\sc J.~Dick, F.Y. Kuo, and I.H. Sloan}, {\em High-dimensional integration:
  {T}he quasi-{Monte Carlo} way}, vol.~22 of Acta Num., Cambridge University
  Press, 2013, pp.~133--288.

\bibitem{dn97}
{\sc C.~R. Dietrich and G.~N. Newsam}, {\em Fast and exact simulation of
  stationary {G}aussian processes through circulant embedding of the covariance
  matrix}, SIAM J. Sci. Comput., 18 (1997), pp.~1088--1107.

\bibitem{kst13}
{\sc T.J. Dodwell, C.~Ketelsen, R.~Scheichl, and A.L. Teckentrup}, {\em A
  hierarchical multilevel {M}arkov chain {M}onte {C}arlo algorithm with
  applications to uncertainty quantification in subsurface flow}, SIAM/ASA J.
  Uncertainty Quantification, 3 (2015), pp.~1075--1108.

\bibitem{dl09}
{\sc P.~Doukhan and G.~Lang}, {\em Evaluation for moments of a ratio with
  application to regression estimation}, Bernoulli, 15 (2009), pp.~1259--1286.

\bibitem{ehm14}
{\sc D.~Elfverson, F.~Hellman, and A.~M{\aa}lqvist}, {\em A multilevel {M}onte
  {C}arlo method for computing failure probabilities}, SIAM/ASA Journal on
  Uncertainty Quantification, 4 (2016), pp.~312--330.

\bibitem{fieller32}
{\sc E.C. Fieller}, {\em The distribution of the index in a normal bivariate
  population}, Biometrika,  (1932), pp.~428--440.

\bibitem{gp12}
{\sc C.~Galeone and A.~Pollastri}, {\em Confidence intervals for the ratio of
  two means using the distribution of the quotient of two normals}, Statistics
  in Transition, 13 (2012), pp.~451--472.

\bibitem{gp16}
{\sc R.~N. Gantner and M.~D. Peters}, {\em Higher order quasi-{M}onte {C}arlo
  for {B}ayesian shape inversion}, Tech. Report 2016-42, Seminar for Applied
  Mathematics, ETH Z{\"u}rich, Switzerland, 2016.

\bibitem{ghanem_spanos}
{\sc R.~G. Ghanem and P.~D. Spanos}, {\em Stochastic finite elements: a
  spectral approach}, Springer, New York, 1991.

\bibitem{giles08}
{\sc M.~B. Giles}, {\em Multilevel {M}onte {C}arlo path simulation}, Oper.
  Res., 256 (2008), pp.~981--986.

\bibitem{gknsss15}
{\sc I.G. Graham, F.Y. Kuo, J.A. Nichols, R.~Scheichl, Ch. Schwab, and I.H.
  Sloan}, {\em Quasi-{M}onte {C}arlo finite element methods for elliptic {PDEs}
  with lognormal random coefficients}, Numer. Math., 131 (2015), pp.~329--368.

\bibitem{gknss11}
{\sc I.~G. Graham, F.~Y. Kuo, D.~Nuyens, R.~Scheichl, and I.~H. Sloan}, {\em
  Quasi-{M}onte {C}arlo methods for elliptic {PDE}s with random coefficients
  and applications}, J. Comput. Phys., 230 (2011), pp.~3668--3694.

\bibitem{gknss15}
\leavevmode\vrule height 2pt depth -1.6pt width 23pt, {\em Circulant embedding
  with {QMC} -- analysis for elliptic {PDE} with lognormal coefficients}.
\newblock In preparation, 2016.

\bibitem{anst14}
{\sc A.-L. Haji-Ali, F.~Nobile, E.~von Schwerin, and R.~Tempone}, {\em
  Optimization of mesh hierarchies in multilevel {M}onte {C}arlo samplers},
  Stochastics and Partial Differential Equations Analysis and Computations, 4
  (2016), pp.~76--112.

\bibitem{heinrich01}
{\sc S.~Heinrich}, {\em Multilevel {M}onte {C}arlo methods}, vol.~2179 of
  Lecture notes in Comput. Sci., Springer, 2001, pp.~3624--3651.

\bibitem{hinkley69}
{\sc D.V. Hinkley}, {\em On the ratio of two correlated normal random
  variables}, Biometrika, 56 (1969), pp.~635--639.

\bibitem{hss13}
{\sc V.~H. Hoang, Ch. Schwab, and A.~M. Stuart}, {\em Complexity analysis of
  accelerated {MCMC} methods for {B}ayesian inversion}, Inverse Prob., 29
  (2013), p.~085010.

\bibitem{kl01}
{\sc C.A. Kennedy and W.C. Lennox}, {\em Moment operations on random variables,
  with applications for probabilistic analysis}, Probabilist. Eng. Mech., 16
  (2001), pp.~253--259.

\bibitem{kssu15}
{\sc F.~Y. Kuo, R.~Scheichl, Ch. Schwab, I.~H. Sloan, and E.~Ullmann}, {\em
  Multilevel quasi-{Monte Carlo} methods for lognormal diffusion problems},
  Preprint, arXiv:1507.01090 (2015).

\bibitem{kss12}
{\sc F.~Y. Kuo, Ch. Schwab, and I.~H. Sloan}, {\em Quasi-{M}onte {C}arlo finite
  element methods for a class of elliptic partial differential equations with
  random coefficients}, SIAM J. Numer. Anal., 50 (2012), pp.~3351--3374.

\bibitem{kss15}
\leavevmode\vrule height 2pt depth -1.6pt width 23pt, {\em Multi-level
  quasi-{Monte Carlo} finite element methods for a class of elliptic {PDEs}
  with random coefficients}, Found.~Comput.~Math., 15 (2015), pp.~411--449.

\bibitem{ks05}
{\sc F.~Y. Kuo and I.~H. Sloan}, {\em Lifting the curse of dimensionality},
  Not. Am. Math. Soc., 52 (2005), pp.~1320--1328.

\bibitem{niederreiter}
{\sc H.~Niederreiter}, {\em {Random Number Generation and quasi-Monte Carlo
  methods}}, SIAM, 1994.

\bibitem{oa94}
{\sc D.~{\"O}ksoy and L.~Aroian}, {\em The quotient of two correlated normal
  variables with applications}, Commun. Stat. Simulat., 23 (1994),
  pp.~223--241.

\bibitem{robert_casella}
{\sc C.~Robert and G.~Casella}, {\em Monte {C}arlo {S}tatistical {M}ethods},
  Springer, 1999.

\bibitem{ss13}
{\sc C.~Schillings and C.~Schwab}, {\em Sparse, adaptive {S}molyak quadratures
  for {B}ayesian inverse problems}, Inverse Probl., 29 (2013), p.~065011.

\bibitem{ss15}
\leavevmode\vrule height 2pt depth -1.6pt width 23pt, {\em Scaling limits in
  computational {B}ayesian inversion}, SAM Research Report, ETH Z\"urich,
  2014-26 (2014).

\bibitem{ss14}
\leavevmode\vrule height 2pt depth -1.6pt width 23pt, {\em Sparsity in
  {B}ayesian inversion of parametric operator equations}, Inverse Probl., 30
  (2014), p.~065007.

\bibitem{ss12}
{\sc Ch. Schwab and A.M. Stuart}, {\em Sparse deterministic approximation of
  {B}ayesian inverse problems}, Inverse Probl., 28 (2012), p.~045003.

\bibitem{stuart10}
{\sc A.~M. Stuart}, {\em Inverse problems: a {B}ayesian perspective}, vol.~19
  of Acta Num., Cambridge University Press, 2010, pp.~451--559.

\bibitem{teckentrup_thesis}
{\sc A.~L. Teckentrup}, {\em Multilevel Monte Carlo methods and uncertainty
  quantification}, PhD thesis, University of Bath, 2013.

\bibitem{tsgu13}
{\sc A.~L. Teckentrup, R.~Scheichl, M.~B. Giles, and E.~Ullmann}, {\em Further
  analysis of multilevel {M}onte {C}arlo methods for elliptic {PDE}s with
  random coefficients}, Numer. Math., 125 (2013), pp.~569--600.

\bibitem{twz15}
{\sc H.~Tran, C.G. Webster, and G.~Zhang}, {\em Analysis of quasi-optimal
  polynomial approximations for parameterized {PDE}s with deterministic and
  stochastic coefficients}, arXiv preprint arXiv:1508.01821,  (2015).

\end{thebibliography}

\end{document}